\input ./style/arxiv-general.cfg
\documentclass[MSNbibl,number,citesort,dvips]{arxbj}
\makeatletter
   \@ifpackageloaded{graphicx}{}{\usepackage{graphicx}}
\makeatother
\usepackage{mathrsfs}

\volume{22}
\issue{4}
\pubyear{2016}
\firstpage{2486}
\lastpage{2520}
\doi{10.3150/15-BEJ735}
\docsubty{FLA}

\makeatletter
\newcommand{\eqref}[1]{(\ref{#1})}
\newtheorem{thmm}{Theorem}[section]
\newtheorem{lem}[thmm]{Lemma}
\newtheorem{cor}[thmm]{Corollary}
\newremark{rmk}[thmm]{Remark}
\newtheorem{prop}[thmm]{Propostion}

\newproclaim{definition}[thmm]{Definition}
\newremark{ex}[thmm]{Example}
\newproclaim{ass}{Assumption}

\newcommand{\dd}{ \mathrm{d}}
\newcommand{\R}{\mathbb{R}}
\newcommand{\F}{\mathcal{F}}
\newcommand{\N}{\mathbb{N}}

\newcommand{\E}{\mathbb{E}}
\newcommand{\1}{\mathbf{1}}
\renewcommand{\P}{\mathbb{P}}
\newcommand{\Q}{\mathbb{Q}}
\newcommand{\Hc}{\mathcal{H}}
\newcommand{\oP}{\overline{P}}
\newcommand{\var}{\operatorname{var}}
\newcommand{\too}{\rightarrow}
\newcommand{\BS}{\mathbb{S}}
\newcommand{\CC}{\mathscr{C}}

\makeatother

\begin{document}
\begin{frontmatter}

\title{Pathwise stochastic integrals for model free~finance}

\runtitle{Pathwise stochastic integrals for model free finance}

\begin{aug}
\author[A]{\inits{N.}\fnms{Nicolas} \snm{Perkowski}\corref{}\ead[label=e1]{perkowski@ceremade.dauphine.fr}}
\and
\author[B]{\inits{D.J.}\fnms{David J.} \snm{Pr\"omel}\ead[label=e2]{proemel@math.hu-berlin.de}}
\address[A]{CEREMADE \& CNRS UMR 7534,
Universit\'e Paris-Dauphine,
France.\\
\printead{e1}}
\address[B]{Humboldt-Universit\"at zu Berlin,
Institut f\"ur Mathematik,
Germany.\\
\printead{e2}}
%
\end{aug}

%
\received{\smonth{11} \syear{2014}}
%
\revised{\smonth{3} \syear{2015}}

%
\begin{abstract}
We present two different approaches to stochastic integration in
frictionless model free financial mathematics. The first one is in the
spirit of It\^o's integral and based on a certain topology which is
induced by the outer measure corresponding to the minimal superhedging
price. The second one is based on the controlled rough path integral.
We prove that every ``typical price path'' has a naturally associated
It\^o rough path, and justify the application of the controlled rough
path integral in finance by showing that it is the limit of
non-anticipating Riemann sums, a new result in itself. Compared to the
first approach, rough paths have the disadvantage of severely
restricting the space of integrands, but the advantage of being a
Banach space theory.

Both approaches are based entirely on financial arguments and do not
require any probabilistic structure.
\end{abstract}

%
\end{frontmatter}

\section{Introduction}\label{sec1}

In this paper, we use Vovk's~\cite{Vovk2012} game-theoretic approach to
develop two different techniques of stochastic integration in
frictionless model free financial mathematics. A priori the integration
problem is highly non-trivial in the model free context since we do not
want to assume any probabilistic, respectively, semimartingale
structure. Therefore, we do not have access to It\^o integration and
most known techniques completely break down. There are only two general
solutions to the integration problem in a non-probabilistic continuous
time setting that we are aware of. One was proposed by~\cite
{Dolinsky2013}, who simply restrict themselves to trading strategies
(integrands) of bounded variation. While this already allows to solve
many interesting problems, it is not a very natural assumption to make
in a frictionless market model. Indeed, in~\cite{Dolinsky2013} a
general duality approach is developed for pricing path-dependent
derivatives that are Lipschitz continuous in the supremum norm, but so
far their approach does not allow to treat derivatives depending on the
volatility.

Another interesting solution was proposed by~\cite{Davis2013} (using an
idea which goes back to \cite{Lyons1995}). They restrict the set of
``possible price paths'' to those admitting a quadratic variation. This
allows them to apply F\"ollmer's pathwise It\^{o} calculus~\cite
{Follmer1981} to define pathwise stochastic integrals of the form $\int
\nabla F(S) \,\dd S$. In~\cite{Lyons1995}, that approach was used to
derive prices for American and European options under volatility
uncertainty. In~\cite{Davis2013}, the given data is a finite number of
European call and put prices and the derivative to be priced is a
weighted variance swap. The restriction to the set of paths with
quadratic variation is justified by referring to Vovk~\cite{Vovk2012},
who proved that ``typical price paths'' (to be defined below) admit a
quadratic variation.

In our first approach, we do not restrict the set of paths and work on
the space $\Omega$ of $d$-dimensional continuous paths (which represent
all possible asset price trajectories). We follow Vovk in introducing
an outer measure on $\Omega$ which is defined as the pathwise minimal
superhedging price (in a suitable sense), and therefore has a purely
financial interpretation and does not come from an artificially imposed
probabilistic structure. Our first observation is that Vovk's outer
measure allows us to define a topology on processes on $\Omega$, and
that the ``natural It\^o integral'' on step functions is in a certain
sense continuous in that topology. This allows us to extend the
integral to c\`adl\`ag adapted integrands, and we call the resulting
integral ``model free It\^o integral''. We stress that the entire
construction is based only on financial arguments.

Let us also stress that it is the \emph{continuity} of our integral
which is the most important aspect. Without reference to any topology,
the construction would certainly not be very useful, since already in
the classical probabilistic setting virtually all applications of the
It\^o integral (SDEs, stochastic optimization, duality theory, \ldots)
are based on the fact that it is a continuous operator.

This also motivates our second approach, which is more in the spirit
of~\cite{Lyons1995,Davis2013,Dolinsky2013}. While in the first approach
we do have a continuous operator, it is only continuous with respect to
a sequence of pseudometrics and it seems impossible to find a Banach
space structure that is compatible with it. This is a pity since Banach
space theory is one of the key tools in the classical theory of
financial mathematics, as emphasized, for example, in \cite{Delbaen2001}.
However, using the model free It\^o integral we are able to show that
every ``typical price path'' has a natural It\^o rough path associated
to it. Since in financial applications we can always restrict ourselves
to typical price paths, this observation opens the door for the
application of the controlled rough path integral~\cite
{Lyons1998,Gubinelli2004} in model free finance. Controlled rough path
integration has the advantage of being an entirely linear Banach space
theory which simultaneously extends:
\begin{itemize}
\item the Riemann--Stieltjes integral of $S$ against functions of
bounded variation which was used by \cite{Dolinsky2013};
\item the Young integral~\cite{Young1936}: typical price paths have
finite $p$-variation for every $p > 2$, and therefore for every $F$ of
finite $q$-variation for $1\leq q < 2$ (so that $1/p + 1/q > 1$), the
integral $\int F \,\dd S$ is defined as limit of non-anticipating Riemann sums;
\item F\"ollmer's~\cite{Follmer1981} pathwise It\^o integral, which was
used by \cite{Lyons1995,Davis2013}. That this last integral is a
special case of the controlled rough path integral is, to the best of
our knowledge, proved rigorously for the first time in this paper,
although also~\cite{Friz2014} contains some observations in that direction.
\end{itemize}
In other words, our second approach covers all previously known
techniques of integration in model free financial mathematics, while
the first approach is much more general but at the price of leaving the
Banach space world.

There is only one pitfall: the rough path integral is usually defined
as a limit of compensated Riemann sums which have no obvious financial
interpretation. This sabotages our entire philosophy of only using
financial arguments. That is why we show that under some weak condition
every rough path integral $\int F \,\dd S$ is given as limit of
non-anticipating Riemann sums that do not need to be compensated -- the
first time that such a statement is shown for general rough path
integrals. Of course, this will not change anything in concrete
applications, but it is of utmost importance from a philosophical point
of view. Indeed, the justification for using the It\^o integral in
classical financial mathematics is crucially based on the fact that it
is the limit of non-anticipating Riemann sums, even if in ``every day
applications'' one never makes reference to that; see, for example, the
discussion in~\cite{Lyons1995}.

\textit{Plan of the paper}.
Below we present a very incomplete list of solutions to the stochastic
integration problem under model uncertainty and in a discrete time
model free context (both a priori much simpler problems than the
continuous time model free case), and we introduce some notations and
conventions that will be used throughout the paper. In Section~\ref{sec:vovk}, we briefly recall Vovk's game-theoretic approach to
mathematical finance and introduce our outer measure. We also construct
a topology on processes which is induced by the outer measure.
Section~\ref{sec:cpi} is devoted to the construction of the model free
It\^o integral. Section~\ref{sec:rough paths} recalls some basic
results from rough path theory, and continues by constructing rough
paths associated to typical price paths. Here we also prove that the
rough path integral is given as a limit of non-anticipating Riemann
sums. Furthermore, we compare F\"ollmer's pathwise It\^o integral with
the rough path integral and prove that the latter is an extension of
the former.
Appendix~\ref{a:hoeffding} recalls Vovk's pathwise Hoeffding
inequality. In Appendix~\ref{a:davie}, we show that a result of Davie
which also allows to calculate rough path integrals as limit of Riemann
sums is a special case of our results in Section~\ref{sec:rough paths}.

\textit{Stochastic integration under model uncertainty}.
The first works which studied the option pricing problem under model
uncertainty were \cite{Avellaneda1995} and \cite{Lyons1995}, both
considering the case of volatility uncertainty. As described above,
\cite{Lyons1995} is using F\"ollmer's pathwise It\^o integral. In \cite
{Avellaneda1995} the problem is reduced to the classical setting by
deriving a ``worst case'' model for the volatility.

A powerful tool in financial mathematics under model uncertainty is
Karandikar's pathwise construction of the It\^o integral~\cite
{Karandikar1995,Bichteler1981}, which allows to construct the It\^o
integral of a c\`adl\`ag integrand simultaneously under all
semimartingale measures. The crucial point that makes the construction
useful is that the It\^o integral is a continuous operator under every
semimartingale measure. While its pathwise definition would allow us to
use the same construction also in a model free setting, it is not even
clear what the output should signify in that case (e.g., the
construction depends on a certain sequence of partitions and changing
the sequence will change the output). Certainly it is not obvious
whether the Karandikar integral is continuous in any topology once we
dispose of semimartingale measures. A more general pathwise
construction of the It\^o integral was given in \cite{Nutz2012}, but it
suffers from the same drawbacks with respect to applications in model
free finance.

A general approach to stochastic analysis under model uncertainty was
put forward in \cite{Denis2006}, and it is based on quasi sure
analysis. This approach is extremely helpful when working under model
uncertainty, but it also does not allow us to define stochastic
integrals in a model free context.

In a related but slightly different direction, in \cite{Coviello2011}
non-semimartingale models are studied (which do not violate arbitrage
assumptions if the set of admissible strategies is restricted). While
the authors work under one fixed probability measure, the fact that
their price process is not a semimartingale prevents them from using
It\^o integrals, a difficulty which is overcome by working with the
Russo--Vallois integral \cite{Russo1993}.

Of course all these technical problems disappear if we restrict
ourselves to discrete time, and indeed in that case \cite
{Beiglbock2013} develop an essentially fully satisfactory duality
theory for the pricing of derivatives under model uncertainty.

\textit{Notation and conventions}.
Throughout the paper, we fix $T \in(0,\infty)$ and we write $\Omega:=
C([0,T], \R^d)$ for the space of $d$-dimensional continuous paths. The
coordinate process on $\Omega$ is denoted by $S_t(\omega) = \omega(t)$,
$t \in[0,T]$. For $i \in\{1, \ldots, d\}$, we also write $S^i_t(\omega
):= \omega^i(t)$, where $\omega= (\omega^1, \ldots, \omega^d)$. The
filtration $(\F_t)_{t \in[0,T]}$ is defined as $\F_t := \sigma(S_s: s
\le t)$, and we set $\F:= \F_T$. Stopping times $\tau$ and the
associated $\sigma$-algebras $\F_\tau$ are defined as usual.

Unless explicitly stated otherwise, inequalities of the type $F_t \ge
G_t$, where $F$ and $G$ are processes on $\Omega$, are supposed to hold
for all $\omega\in\Omega$, and not modulo null sets, as it is usually
assumed in stochastic analysis.

The indicator function of a set $A$ is denoted by $\1_A$.

A \emph{partition} $\pi$ of $[0,T]$ is a finite set of time points,
$\pi= \{0 = t_0 < t_1 < \cdots< t_m = T\}$. Occasionally, we will
identify $\pi$ with the set of intervals $\{[t_0,t_1], [t_1, t_2], \ldots
, [t_{m-1}, t_m]\}$, and write expressions like $\sum_{[s,t] \in\pi}$.

For $f\colon[0,T] \too\R^n$ and $t_1,t_2 \in[0,T]$, denote
$f_{t_1,t_2} := f(t_2) - f(t_1)$ and define the $p$-variation of $f$
restricted to $[s,t] \subseteq[0,T]$ as
%
\begin{equation}
\label{eq:def p-var} \Vert f \Vert_{p\mbox{-}\var,[s,t]} := \sup \Biggl\{ \Biggl(\sum
_{k=0}^{m-1} |f_{t_{k},t_{k+1}}|^p
\Biggr)^{1/p}: s = t_0 < \cdots< t_m = t, m \in
\N \Biggr\},\qquad p>0,
\end{equation}
(possibly taking the value $+\infty$). We set $\| f \|_{p\mbox{-}\var}
:= \| f \|_{p\mbox{-}\var,[0,T]}$. We write $\Delta_T := \{(s,t): 0 \le
s \le t \le T\}$ for the simplex and define the $p$-variation of a
function $g\colon\Delta_T \too\R^n$ in the same manner, replacing
$f_{t_{k},t_{k+1}}$ in~\eqref{eq:def p-var} by $g(t_{k},t_{k+1})$.

For $\alpha>0$ and $\lfloor\alpha\rfloor:= \max\{z\in\mathbb{Z} :
z\leq\alpha\}$, the space $C^\alpha$ consists of those functions that
are $\lfloor\alpha\rfloor$ times continuously differentiable, with
$(\alpha- \lfloor\alpha\rfloor)$-H\"older continuous partial
derivatives of order $\lfloor\alpha\rfloor$ (and with continuous
partial derivatives of order $\alpha$ in case $\alpha= \lfloor\alpha
\rfloor$). The space $C^\alpha_b$ consists of those functions in
$C^\alpha$ that are bounded, together with their partial derivatives,
and we define the norm $\Vert\cdot\Vert_{C^\alpha_b}$ by setting
\[
\Vert f \Vert_{C^\alpha_b} := \sum_{k=0}^{\lfloor\alpha\rfloor}
\bigl\| D^k f \bigr\|_\infty+ \1_{\alpha>\lfloor\alpha\rfloor} \bigl\| D^{\lfloor
\alpha\rfloor} f
\bigr\|_{\alpha- \lfloor\alpha\rfloor},
\]
where $\Vert\cdot\Vert_\beta$ denotes the $\beta$-H\"older norm for
$\beta\in(0,1)$, and $\Vert\cdot\Vert_\infty$ denotes the
supremum norm.

For $x,y \in\R^d$, we write $xy := \sum_{i=1}^d x_i y_i$ for the usual
inner product. However, often we will encounter terms of the form $\int
S \,\dd S$ or $S_s S_{s,t}$ for $s,t \in[0,T]$, where we recall that $S$
denotes the coordinate process on $\Omega$. Those expressions are to be
understood as the matrix $(\int S^i \,\dd S^j)_{1 \le i,j \le d}$, and
similarly for $S_s S_{s,t}$. The interpretation will be usually clear
from the context, otherwise we will make a remark to clarify things.

We use the notation $a \lesssim b$ if there exists a constant $c > 0$,
independent of the variables under consideration, such that $a \le c
\cdot b$, and we write $a\simeq b$ if $a \lesssim b$ and $b \lesssim
a$. If we want to emphasize the dependence of $c$ on the variable $x$,
then we write $a(x) \lesssim_x b(x)$.

We make the convention that $0/0 := 0\cdot\infty:= 0$, $1\cdot\infty
:= \infty$ and $\inf\varnothing:= \infty$.

\section{Superhedging and typical price paths}\label{sec:vovk}

\subsection{The outer measure and its basic properties}

In a recent series of papers, Vovk \cite{Vovk2008,Vovk2011,Vovk2012}
has introduced a model free, hedging based approach to mathematical
finance that uses arbitrage considerations
to examine which properties are satisfied by ``typical price paths''.
This is achieved with the help of an outer measure given by the
cheapest superhedging price.

Recall that $T \in(0,\infty)$ and $\Omega= C([0,T], \R^d)$ is the
space of continuous paths, with coordinate process $S$, natural
filtration $(\F_t)_{t \in[0,T]}$, and $\F= \F_T$. A process $H\colon
\Omega\times[0,T] \rightarrow\R^d$ is called a \emph{simple
strategy} if there exist stopping times $0 = \tau_0 < \tau_1 < \cdots,$
and $\F_{\tau_n}$-measurable bounded functions $F_n\colon\Omega
\rightarrow\R^d$, such that for every $\omega\in\Omega$ we have $\tau
_n(\omega) = \infty$ for all
but finitely many $n$, and such that
\[
H_t(\omega) = \sum_{n=0}^\infty
F_n(\omega) \1_{(\tau_n(\omega),\tau
_{n+1}(\omega)]}(t).
\]
In that case, the integral
\[
(H \cdot S)_t(\omega) := \sum_{n=0}^\infty
F_n(\omega) \bigl(S_{\tau_{n+1}
\wedge t}(\omega) - S_{\tau_n\wedge t}(
\omega)\bigr) = \sum_{n=0}^\infty
F_n(\omega) S_{\tau_n\wedge t, \tau_{n+1} \wedge t}(\omega)
\]
is well defined for all $\omega\in\Omega$, $t \in[0,T]$. Here
$F_n(\omega) S_{\tau_n\wedge t, \tau_{n+1} \wedge t}(\omega)$ denotes the
usual inner product on $\R^d$. For $\lambda> 0$, a simple strategy $H$
is called $\lambda$-\textit{admissible} if $(H\cdot S)_t(\omega) \ge-
\lambda$ for all $\omega\in\Omega$, $t \in[0,T]$. The set of $\lambda
$-admissible simple strategies is denoted by $\mathcal{H}_\lambda$.

\begin{definition}
The \emph{outer measure} of $A \subseteq\Omega$ is defined as the
cheapest superhedging price for $\1_A$, that is
\[
\overline{P}(A) := \inf \Bigl\{\lambda> 0: \exists \bigl(H^n
\bigr)_{n\in\N} \subseteq\mathcal{H}_\lambda\mbox{ s.t. } \liminf
_{n\rightarrow\infty
} \bigl(\lambda+ \bigl(H^n\cdot S
\bigr)_T(\omega)\bigr) \ge\1_A(\omega)\ \forall\omega
\in\Omega \Bigr\}.
\]
A set of paths $A \subseteq\Omega$ is called a \emph{null set} if it
has outer measure zero.
\end{definition}

The term outer measure will be justified by Lemma~\ref{lem:outer}
below. Our definition of $\oP$ is very similar to the one used by
Vovk~\cite{Vovk2012}, but not quite the same. For a discussion, see
Section~\ref{ss:relation to vovk} below.

By definition, every It\^{o} stochastic integral is the limit of
stochastic integrals against simple strategies. Therefore, our
definition of the cheapest superhedging price is essentially the same
as in the classical setting, with one important difference: we require
superhedging for \emph{all} $\omega\in\Omega$, and not just almost surely.

\begin{rmk}[(\cite{Vovk2012}, page~564)]\label{r:vovk take supremum in content}
An equivalent definition of $\overline{P}$ would be
\[
\widetilde{P}(A) := \inf \Bigl\{\lambda> 0: \exists \bigl(H^n
\bigr)_{n\in\N} \subseteq\mathcal{H}_\lambda\mbox{ s.t. } \liminf
_{n\rightarrow\infty
} \sup_{t \in[0,T]} \bigl(\lambda+
\bigl(H^n\cdot S\bigr)_t(\omega)\bigr) \ge
\1_A(\omega )\ \forall\omega\in\Omega \Bigr\}.
\]
Clearly, $\widetilde{P} \le\overline{P}$. To see the opposite
inequality, let $\widetilde{P}(A) < \lambda$. Let $(H^n)_{n\in\N}
\subset\mathcal{H}_\lambda$ be a sequence of
simple strategies such that $\liminf_{n\rightarrow\infty} \sup_{t \in
[0,T]} (\lambda+ (H^n\cdot S)_t) \ge\1_A$, and let $\varepsilon> 0$.
Define $\tau_n:= \inf\{t \in[0,T]: \lambda+ \varepsilon+ (H^n \cdot
S)_t \ge1\}$. Then the stopped strategy $G^n_t(\omega) := H^n_t(\omega
) \1_{[0, \tau_n(\omega))}(t)$
is in $\mathcal{H}_{\lambda} \subseteq\mathcal{H}_{\lambda+
\varepsilon}$ and
\[
\liminf_{n \rightarrow\infty} \bigl(\lambda+ \varepsilon+
\bigl(G^n\cdot S\bigr)_T(\omega)\bigr) \ge\liminf
_{n \rightarrow\infty} \1_{\{\lambda+ \varepsilon+ \sup
_{t \in[0,T]} (H^n\cdot S)_t \ge1\}}(\omega) \ge\1_A(
\omega).
\]
Therefore $\overline{P}(A)\le\lambda+ \varepsilon$, and since
$\varepsilon> 0$ was arbitrary $\overline{P} \le\widetilde{P}$, and
thus $\overline{P} = \widetilde{P}$.
\end{rmk}

\begin{lem}[(\cite{Vovk2012}, Lemma~4.1)]\label{lem:outer}
$\overline{P}$ is in fact an outer measure, that is, a non-negative
function defined on the subsets of $\Omega$ such that
\begin{longlist}[--]
\item[--] $\overline{P}(\varnothing) = 0$;
\item[--] $\overline{P}(A) \le\overline{P}(B)$ if $A \subseteq B$;
\item[--] if $(A_n)_{n\in\N}$ is a sequence of subsets of $\Omega$, then
$\overline{P}(\bigcup_n A_n) \le\sum_n \overline{P}(A_n)$.
\end{longlist}
\end{lem}

\begin{pf}
Monotonicity and $\overline{P}(\varnothing)=0$ are obvious. So let
$(A_n)$ be a sequence of subsets of $\Omega$.
Let $\varepsilon> 0$, $n \in\N$, and let $(H^{n,m})_{m\in\N}$
be a sequence of $(\overline{P}(A_n) + \varepsilon
2^{-n-1})$-admissible simple strategies such that
$\liminf_{m\rightarrow\infty} (\overline{P}(A_n) + \varepsilon
2^{-n-1} + (H^{n,m}\cdot S)_T) \ge\1_{A_n}$. Define for $m \in\N$ the
$(\sum_n \overline{P}(A_n) + \varepsilon)$-admissible
simple strategy $G^m := \sum_{n=0}^m H^{n,m}$. Then by Fatou's lemma
\begin{eqnarray*}
\liminf_{m \rightarrow\infty} \Biggl(\sum_{n=0}^\infty
\overline {P}(A_n) + \varepsilon+ \bigl(G^m \cdot S
\bigr)_T \Biggr) & \ge&\sum_{n=0}^k
\Bigl( \overline{P}(A_n) + \varepsilon2^{-n-1} + \liminf
_{m \rightarrow\infty} \bigl(H^{n,m} \cdot S\bigr)_T
\Bigr)
\\
& \ge&\1_{\bigcup_{n=0}^k A_n}
\end{eqnarray*}
for all $k \in\N$. Since the left-hand side does not depend on $k$, we
can replace $\1_{\bigcup_{n=0}^k A_n}$ by $\1_{\bigcup_n A_n}$ and the
proof is complete.
\end{pf}

Maybe the most important property of $\overline{P}$ is that there
exists an arbitrage interpretation for sets with outer measure zero.

\begin{lem}
A set $A \subseteq\Omega$ is a null set if and only if there exists a
sequence of 1-admissible simple strategies $(H^n)_n \subset\mathcal
{H}_1$ such that
%
\begin{equation}
\label{e:NA1 with simple strategies} \liminf_{n \rightarrow\infty} \bigl(1 + \bigl(H^n
\cdot S\bigr)_T(\omega)\bigr) \ge \infty\cdot\1_A(
\omega),
\end{equation}
where we use the convention $0 \cdot\infty= 0$ and $1\cdot\infty:=
\infty$.
\end{lem}

\begin{pf}
If such a sequence exists, then we can scale it down by an arbitrary
factor $\varepsilon> 0$ to obtain a sequence of strategies in $\mathcal
{H}_\varepsilon$ that superhedge $\1_A$, and therefore $\overline{P}(A)
= 0$.

If conversely $\overline{P}(A) = 0$, then for every $n\in\N$ there
exists a sequence of simple strategies $(H^{n,m})_{m \in\N} \subset
\mathcal{H}_{2^{-n-1}}$
such that $2^{-n-1} + \liminf_{m \rightarrow\infty} (H^{n,m}\cdot
\omega)_T \geq\1_A(\omega)$ for all $\omega\in\Omega$. Define $G^m
:= \sum_{n=0}^m H^{n,m}$,
so that $G^m \in\mathcal{H}_1$. For every $k \in\N$, we obtain
\[
\liminf_{m\rightarrow\infty} \bigl(1 + \bigl(G^m \cdot S
\bigr)_T \bigr) \ge\sum_{n=0}^k
\Bigl(2^{-n-1} + \liminf_{m \rightarrow\infty} \bigl(H^{n,m}
\cdot S\bigr)_T \Bigr) \ge(k+1) \1_A.
\]
Since the left-hand side does not depend on $k$, the sequence $(G^m)$
satisfies \eqref{e:NA1 with simple strategies}.
\end{pf}

In other words, if a set $A$ has outer measure $0$, then we can make
infinite profit by investing in the paths from $A$, without ever
risking to lose more than the initial capital 1.

This motivates the following definition.

\begin{definition}\label{def:typical}
We say that a property (P) holds for \emph{typical price paths} if the
set $A$ where (P) is violated is a null set.
\end{definition}

The basic idea of Vovk, which we shall adopt in the following, is that
we only need to concentrate on typical price paths. Indeed,
``non-typical price paths'' can be excluded since they are in a certain
sense ``too good to be true'': they would allow investors to realize
infinite profit while at the same time taking essentially no risk.

\subsection{Arbitrage notions and link to classical mathematical finance}

Before we continue, let us discuss different notions of arbitrage and
link our outer measure to classical mathematical finance. We start by
observing that $\overline{P}$ is an outer measure which simultaneously
dominates all local martingale measures on $\Omega$.

\begin{prop}[(\cite{Vovk2012}, Lemma~6.3)]\label{prop:local martingale}
Let $\P$ be a probability measure on $(\Omega, \F)$, such that the
coordinate process $S$ is a $\P$-local martingale, and let $A \in\F$.
Then $\P(A) \le\overline{P}(A)$.
\end{prop}

\begin{pf}
Let $\lambda> 0$ and let $(H^n)_{n \in\mathbb{N}} \subseteq\mathcal
{H}_\lambda$ be such that $ \liminf_n(\lambda+ (H^n \cdot S)_T) \ge\1
_A$. Then
\[
\P(A) \le\E_\P\Bigl[\liminf_n \bigl(\lambda+
\bigl(H^n \cdot S\bigr)_T\bigr)\Bigr] \le\liminf
_n \E _\P\bigl[\lambda+ \bigl(H^n
\cdot S\bigr)_T\bigr] \le\lambda,
\]
where in the last step we used that $\lambda+ (H^n \cdot S)$ is a
non-negative $\P$-local martingale and thus a $\P$-supermartingale.
\end{pf}

This already indicates that $\overline P$-null sets are quite
degenerate, in the sense that they are null sets under all local
martingale measures. However, if that was the only reason for our
definition of typical price paths, then a definition based on model
free arbitrage opportunities would be equally valid. A map $X\colon
\Omega\to[0,\infty)$ is a \emph{model free arbitrage opportunity} if
$X$ is not identically 0 and if there exists $c>0$ and a sequence
$(H^n) \subseteq\mathcal H_c$ such that $\liminf_{n \to\infty} (H^n
\cdot S)_T(\omega) = X(\omega)$ for all $\omega\in\Omega$. See \cite
{Davis2007,Acciaio2013} where (a similar) definition is used in the
discrete time setting.

It might then appear more natural to say that a property holds for
typical price paths if the indicator function of its complement is a
model free arbitrage opportunity, rather than working with
Definition~\ref{def:typical}. This ``arbitrage definition'' would also
imply that any property which holds for typical price paths is almost
surely satisfied under every local martingale measure. Nonetheless, we
decidedly claim that our definition is ``the correct one''. First of
all, the arbitrage definition would make our life much more difficult
since it seems not very easy to work with. But of course this is only a
convenience and cannot serve as justification of our approach. Instead,
we argue by relating the two notions to classical mathematical finance.

For that purpose, recall the fundamental theorem of asset pricing \cite
{Delbaen1994}: If $\P$ is a probability measure on $(\Omega,\F)$ under
which $S$ is a semimartingale, then there exists an equivalent measure
$\Q$ such that $S$ is a $\Q$-local martingale if and only if $S$ admits
\emph{no free lunch with vanishing risk} (\textit{NFLVR}). But (NFLVR) is
equivalent to the two conditions \emph{no arbitrage} (\textit{NA}) (intuitively:
no profit without risk) and \emph{no arbitrage opportunities of the
first kind} (\textit{NA1}) (intuitively: no very large profit with a small
risk). The (NA) property holds if for every $c > 0$ and every sequence
$(H^n) \subseteq\mathcal H_c$ for which $\lim_{n\to\infty} (H^n \cdot
S)_T(\omega)$ exists for all $\omega$ we have $\P(\lim_{n\to\infty}
(H^n \cdot S)_T < 0) > 0$ or $\P(\lim_{n\to\infty} (H^n \cdot S)_T = 0)
= 1$. The (NA1) property holds if $\{1+ (H \cdot S)_T: H \in\mathcal
{H}_1\}$ is bounded in $\P$-probability, that is, if
\[
\lim_{c \too\infty} \sup_{ H \in\mathcal{H}_1} \P\bigl(1+(H\cdot
S)_T \ge c\bigr) = 0.
\]
Strictly speaking this is (NA1) with simple strategies, but as observed
by \cite{Kardaras2011b} (NA1) and (NA1) with simple strategies are
equivalent; see also \cite{Ankirchner2005,Imkeller2011}.

It turns out that the arbitrage definition of typical price paths
corresponds to (NA), while our definition corresponds to (NA1).
%
\begin{prop}\label{p:model free NA1 vs classical NA1}
Let $A \in\F$ be a null set, and let $\P$ be a probability measure on
$(\Omega, \F)$ such that the coordinate process satisfies (NA1). Then
$\P(A) = 0$.
\end{prop}

\begin{pf}
Let $(H^n)_{n\in\N} \subseteq\mathcal{H}_1$ be such that $1+\liminf_n (H^n \cdot S)_T \ge\infty\cdot\1_A$. Then for every $c > 0$
\begin{eqnarray*}
\P(A) = \P \Bigl(A \cap \Bigl\{\liminf_{n \rightarrow\infty}
\bigl(H^n\cdot S\bigr)_T > c \Bigr\} \Bigr) \le\sup
_{H \in\mathcal{H}_{1}} \P\bigl( \bigl\{(H \cdot S)_T > c\bigr\}
\bigr).
\end{eqnarray*}
By assumption, the right-hand side converges to 0 as $c \rightarrow
\infty$ and thus $\P(A) = 0$.
\end{pf}

\begin{rmk}
Proposition~\ref{p:model free NA1 vs classical NA1} is actually a
consequence of Proposition~\ref{prop:local martingale}, because if $S$
satisfies (NA1) under $\P$, then there exists a dominating measure $\Q
\gg\P$, such that $S$ is a $\Q$-local martingale. See \cite{Ruf2013}
for the case of continuous $S$, and~\cite{Imkeller2011} for the general case.
\end{rmk}

The crucial point is that (NA1) is \emph{the} essential property which
every sensible market model has to satisfy, whereas (NA) is nice to
have but not strictly necessary. Indeed, (NA1) is equivalent to the
existence of an unbounded utility function such that the maximum
expected utility is finite \cite{Karatzas2007,Imkeller2011}. (NA) is
what is needed in addition to (NA1) in order to obtain equivalent local
martingale measures \cite{Delbaen1994}. But there are perfectly viable
models which violate (NA), for example, the three dimensional Bessel
process \cite{Delbaen1995,Karatzas2007}. By working with the arbitrage
definition of typical price paths, we would in a certain sense ignore
these models.

\subsection{Relation to Vovk's outer measure}\label{ss:relation to vovk}

Our definition of the outer measure $\overline{P}$ is not exactly the
same as Vovk's~\cite{Vovk2012}. We find our definition more intuitive
and it also seems to be easier to work with. However, since we rely on
some of the results established by Vovk, let us compare the two notions.

For $\lambda> 0$, Vovk defines the set of processes
\[
\mathcal{S}_\lambda:= \Biggl\{\sum_{k=0}^\infty
H^k: H^k \in\mathcal {H}_{\lambda_k},
\lambda_k > 0, \sum_{k=0}^\infty
\lambda_k = \lambda \Biggr\}.
\]
For every $G = \sum_{k\ge0} H^k \in\mathcal{S}_{\lambda}$, every
$\omega\in\Omega$ and every $t \in[0,T]$, the integral
\[
(G\cdot S)_t(\omega) := \sum_{k \ge0}
\bigl(H^k\cdot S\bigr)_t(\omega) = \sum
_{k
\ge0} \bigl(\lambda_k + \bigl(H^k
\cdot S\bigr)_t(\omega)\bigr) - \lambda
\]
is well defined and takes values in $[-\lambda, \infty]$. Vovk then
defines for $A \subseteq\Omega$ the cheapest superhedging price as
\[
\overline{Q}(A) := \inf \bigl\{\lambda> 0: \exists G \in\mathcal
{S}_\lambda\mbox{ s.t. } \lambda+ (G\cdot S)_T \ge
\1_A \bigr\}.
\]
This definition corresponds to the usual construction of an outer
measure from an outer content (i.e., an outer measure which is only
finitely subadditive and not countably subadditive); see~\cite
{Folland1999}, Chapter~1.4, or~\cite{Tao2011}, Chapter~1.7. Here, the
outer content is given by the cheapest superhedging price using only
simple strategies. It is easy to see that $\overline{P}$ is dominated
by $\overline{Q}$.

\begin{lem}\label{l:vovk content versus ours}
Let $A \subseteq\Omega$. Then $\overline{P}(A) \le\overline{Q}(A)$.
\end{lem}

\begin{pf}
Let $G = \sum_k H^k$, with $H^k \in\mathcal{H}_{\lambda_k}$ and $\sum_k \lambda_k = \lambda$, and assume that $\lambda+ (G\cdot S)_T \ge\1_A$.
Then $(\sum_{k=0}^n H^k)_{n\in\N}$ defines a sequence of simple
strategies in $\mathcal{H}_\lambda$, such that
\[
\liminf_{n\rightarrow\infty} \Biggl(\lambda+ \Biggl( \Biggl(\sum
_{k=0}^n H^k \Biggr) \cdot S
\Biggr)_T \Biggr) = \lambda+ (G\cdot S)_T \ge
\1_A.
\]
So if $\overline{Q}(A) < \lambda$, then also $\overline{P}(A) \le
\lambda$, and therefore $\overline{P}(A) \le\overline{Q}(A)$.
\end{pf}

\begin{cor}\label{cor:typical price paths have finite p-variation}
For every $p > 2$, the set $A_p := \{\omega\in\Omega: \Vert S(\omega
) \Vert_{p\mbox{-}\mathrm{var}} = \infty\}$ has outer measure zero,
that is $\overline{P}(A_p) = 0$.
\end{cor}

\begin{pf}
Theorem~1 of Vovk~\cite{Vovk2008} states that $\overline{Q}(A_p) = 0$,
so $\overline{P}(A_p) = 0$ by Lemma~\ref{l:vovk content versus ours}.
\end{pf}

It is a remarkable result of \cite{Vovk2012} that if $\Omega=
C([0,\infty), \R)$ (i.e., if the asset price process is
one-dimensional), and if $A \subseteq\Omega$
is ``invariant under time changes'' and such that $S_0(\omega) = 0$ for
all $\omega\in A$, then $A \in\F$ and $\overline{Q}(A) = \P(A)$,
where $\P$ denotes the Wiener measure.
This can be interpreted as a pathwise Dambis Dubins--Schwarz theorem.

\subsection{A topology on path-dependent functionals}\label{sec:topology}

It will be very useful to introduce a topology on functionals on $\Omega
$. For that purpose let us identify $X,Y\colon\Omega\to\R$ if $X=Y$
for typical price paths. Clearly this defines an equivalence relation,
and we write $\overline L_0$ for the space of equivalence classes. We
then introduce the analog of convergence in probability in our context:
$(X_n)$ \emph{converges in outer measure} to $X$ if
\[
\lim_{n \to\infty} \overline P\bigl(|X_n - X| > \varepsilon\bigr) =
0 \qquad\mbox{for all } \varepsilon> 0.
\]
We follow~\cite{Vovk2012} in defining an expectation operator. If
$X\colon\Omega\to[0,\infty]$, then
\[
\overline E[X] := \inf \Bigl\{\lambda> 0: \exists \bigl(H^n
\bigr)_{n\in\N} \subseteq\mathcal{H}_\lambda\mbox{ s.t. } \liminf
_{n\rightarrow\infty
} \bigl(\lambda+ \bigl(H^n\cdot S
\bigr)_T(\omega)\bigr) \ge X(\omega)\ \forall\omega\in \Omega \Bigr
\}.
\]
In particular, $\overline P(A) = \overline E[\1_A]$. The expectation
$\overline E$ is countably subadditive, monotone, and positively
homogeneous. It is an easy exercise to verify that
\[
d(X,Y) := \overline E\bigl[|X-Y| \wedge1\bigr]
\]
defines a metric on $\overline L_0$.

\begin{lem}\label{lem:complete metric}
The distance $d$ metrizes the convergence in outer measure. More
precisely, a sequence $(X_n)$ converges to $X$ in outer measure if and
only if $\lim_n d(X_n,X)=0$. Moreover, $(\overline L_0,d)$ is a
complete metric space.
\end{lem}

\begin{pf}
The arguments are the same as in the classical setting. Using
subadditivity and monotonicity of the expectation operator, we have
\[
\varepsilon\overline P\bigl(|X_n-X|\ge\varepsilon\bigr) \le\overline
E\bigl[|X_n-X|\wedge1\bigr] \le\overline P\bigl(|X_n - X| > \varepsilon\bigr)
+ \varepsilon
\]
for all $\varepsilon\in(0,1]$, showing that convergence in outer
measure is equivalent to convergence with respect to $d$.

As for completeness, let $(X_n)$ be a Cauchy sequence with respect to
$d$. Then there exists a subsequence $(X_{n_k})$ such that
$d(X_{n_k},X_{n_{k+1}}) \le2^{-k}$ for all $k$, so that
\[
\overline E \biggl[\sum_k\bigl (|X_{n_k} -
X_{n_{k+1}}| \wedge1\bigr) \biggr] \le\sum_k
\overline E\bigl[|X_{n_k} - X_{n_{k+1}}| \wedge1\bigr] = \sum
_k d(X_{n_k},X_{n_{k+1}}) < \infty,
\]
which means that $(X_{n_k})$ converges for typical price paths. Define
$X:=\liminf_k X_{n_k}$. Then we have for all $n$ and $k$
\[
d(X_n,X) \le d(X_n, X_{n_k}) +
d(X_{n_k},X) \le d(X_n,X_{n_k}) + \sum
_{\ell\ge k} d(X_{n_\ell},X_{n_{\ell+1}}) \le
d(X_n,X_{n_k}) + 2^{-k}.
\]
Choosing $n$ and $k$ large, we see that $d(X_n, X)$ tends to 0.
\end{pf}

\section{Model free It\^o integration}\label{sec:cpi}

The present section is devoted to the construction of a model free It\^
o integral. The main ingredient is a (weak) type of model free It\^o
isometry, which allows us to estimate the integral against a step
function in terms of the amplitude of the step function and the
quadratic variation of the price path. Using the topology introduced in
Section~\ref{sec:topology}, it is then easy to extend the integral to
c\`adl\`ag integrands by a continuity argument.

Since we are in an unusual setting, let us spell out the following
standard definitions.

\begin{definition}
A process $F \colon\Omega\times[0,T] \to\mathbb{R}^d $ is called
\emph{adapted} if the random variable $ \omega\mapsto F_{t}(\omega)$
is $\mathcal{F}_{t}$-measurable for all $t \in[0,T]$.

The process $F$ is said to be \emph{c\`adl\`ag} if the sample path $t
\mapsto F_t(\omega)$ is c\`adl\`ag for all $\omega\in\Omega$.
\end{definition}

To prove our weak It\^o isometry, we will need an appropriate sequence
of stopping times:
Let $n \in\mathbb{N}$. For each $i=1,\ldots,d$ define inductively
\[
\sigma^{n,i}_0 := 0,\qquad \sigma^{n,i}_{k+1}
:= \inf \bigl\{ t \geq \sigma^{n,i}_{k} :
\bigl|S^i_t -S^i_{\sigma^{n,i}_{k}}\bigr|
\ge2^{-n} \bigr\},\qquad k \in\N.
\]
Since we are working with continuous paths and we are considering
entrance times into closed sets, the maps $(\sigma^{n,i})$ are indeed
stopping times, despite the fact that $(\F_t)$ is neither complete nor
right-continuous. Denote $\pi^{n,i} := \{ \sigma^{n,i}_k : k \in\mathbb
{N} \}$. To obtain an increasing sequence of partitions, we take the
union of the $(\pi^{n,i})$, that is we define $\sigma^n_0 := 0$ and then
%
\begin{equation}
\label{eq:dyadic stopping times} \sigma^n_{k+1}(\omega):= \min \Biggl\{ t >
\sigma^n_k(\omega) : t \in \bigcup
_{i=1}^d \pi^{n,i}(\omega) \Biggr\},\qquad k \in
\N,
\end{equation}
and we write $\pi^n := \{\sigma^n_k: k \in\N\}$ for the corresponding
partition.

\begin{lem}[(\cite{Vovk2011}, Theorem~4.1)]
For typical price paths $\omega\in\Omega$, the quadratic variation
along $(\pi^{n,i}(\omega))_{n \in\mathbb{N}}$ exists. That is,
\[
V^{n,i}_t(\omega) := \sum_{k=0}^{\infty}
\bigl(S^i_{\sigma^{n,i}_{k+1}
\wedge t}(\omega)-S^i_{\sigma^{n,i}_{k} \wedge t}(
\omega) \bigr)^2, \qquad t \in[0,T], n \in\mathbb{N},
\]
converges uniformly to a function $\langle S^i \rangle(\omega) \in
C([0,T],\mathbb{R})$ for all $i \in\{1,\ldots,d\}$.
\end{lem}

For later reference, let us estimate $N^n_t := \max\{k \in\N: \sigma
^n_k \leq t \mbox{ and }\sigma^n_k \neq0\}$, the number of stopping
times $\sigma^n_k \neq0$ in $\pi^n$ with values in $[0,t]$:

\begin{lem}\label{lem:number}
For all $\omega\in\Omega$, $n \in\mathbb{N}$, and $t \in[0,T]$, we have
\[
2^{-2n} N_t^n(\omega) \leq \sum
_{i =1 }^d V^{n,i}_t(\omega) =:
V^n_t(\omega).
\]
\end{lem}

\begin{pf}
For $i \in\{1,\ldots, d\}$ define $N_t^{n,i}:=\max\{k \in\N: \sigma
^{n,i}_k \leq t \mbox{ and }\sigma^{n,i}_k \neq0\}$.
Since $S^i$ is continuous, we have $\vert S^i_{\sigma^{n,i}_{k+1}}-
S^i_{\sigma^{n,i}_{k}}\vert= 2^{-n}$ as long as $\sigma^{n,i}_{k+1} \le
T$. Therefore, we obtain
\begin{eqnarray*}
N^n_t(\omega) \le\sum_{i=1}^d
N_t^{n,i}(\omega)  = \sum_{i=1}^d
\sum_{k=0}^{N_t^{n,i}(\omega)-1} \frac{1}{2^{-2n}} \bigl(
S_{\sigma
^{n,i}_{k+1}}(\omega)- S_{\sigma^{n,i}_{k}}(\omega) \bigr)^2
\leq2^{2n} \sum_{i=1}^d
V^{n,i}_t(\omega).
\end{eqnarray*}
\upqed\end{pf}

We will start by constructing the integral against step functions,
which are defined similarly as simple strategies, except possibly
unbounded: A process $F\colon\Omega\times[0,T] \rightarrow\R^d$ is
called a \emph{step function} if there exist stopping times $0 = \tau_0
< \tau_1 < \cdots,$
and $\F_{\tau_n}$-measurable functions $F_n \colon\Omega\rightarrow\R
^d$, such that for every $\omega\in\Omega$ we have $\tau_n(\omega) =
\infty$ for all
but finitely many $n$, and such that
\[
F_t(\omega) = \sum_{n=0}^\infty
F_n(\omega) \1_{[\tau_n(\omega),\tau
_{n+1}(\omega))}(t).
\]
For notational convenience, we are now considering the interval $[\tau
_n(\omega),\tau_{n+1}(\omega))$ which is closed on the left-hand side.
This allows us define the integral
\[
(F\cdot S)_t := \sum_{n=0}^\infty
F_n S_{\tau_n \wedge t, \tau_{n+1}
\wedge t} = \sum_{n=0}^\infty
F_{\tau_n} S_{\tau_n \wedge t, \tau_{n+1}
\wedge t},\qquad t \in[0,T].
\]

The following lemma will be the main building block in the construction
of our integral.

\begin{lem}[(Model free version of It\^o's isometry)]\label{lem:model
free ito isometry}
Let $F$ be a step function. Then for all $a,b,c > 0$ we have
\[
\oP \bigl(\bigl\{\bigl\Vert(F\cdot S)\bigr\Vert_\infty\ge a b \sqrt{c}\bigr\} \cap
\bigl\{ \Vert F\Vert_\infty\le a \bigr\}\cap\bigl\{ \langle S\rangle_T
\le c \bigr\} \bigr) \le2\exp\bigl(-b^2/(2d)\bigr),
\]
where the set $\{ \langle S\rangle_T \le c \}$ should be read as $\{
\langle S \rangle_T = \lim_n V^n_T \mbox{ exists and satisfies }
\langle S \rangle_T \le c\}$.
\end{lem}

\begin{pf}
Assume $F_t = \sum_{n=0}^\infty F_n \1_{[\tau_n,\tau_{n+1})}(t)$ and
set $\tau_a := \inf\{ t > 0 :  |F_t|\geq a \}$. Let $n \in\N$ and
define $\rho^n_0 := 0$ and then for $k \in\N$
\[
\rho^n_{k+1} := \min \bigl\{ t > \rho^n_k:
t \in\pi^n \cup\{\tau_m : m \in\N\} \bigr\},
\]
where we recall that $\pi^n =\{\sigma^n_k: k \in\N\}$ is the $n$th
generation of the dyadic partition generated by $S$.
For $t \in[0,T]$, we have $(F\cdot S)_{\tau_a\wedge t} = \sum_{k}
F_{\rho^n_k} S_{\tau_a \wedge\rho^n_k\wedge t, \tau_a \wedge\rho
^n_{k+1} \wedge t}$, and by the definition of $\pi^n(\omega)$ and $\tau
_a$ we get
\[
\sup_{t\in[0,T]} \bigl| F_{\rho^n_k} S_{\tau_a \wedge\rho^n_k\wedge t,
\tau_a \wedge\rho^n_{k+1} \wedge t} \bigr| \le a
\sqrt{d} 2^{-n}.
\]
Hence, the pathwise Hoeffding inequality, Lemma~\ref{l:hoeffding} in
Appendix~\ref{a:hoeffding}, yields for every $\lambda\in\R$ the
existence of a 1-admissible simple strategy $H^{\lambda,n} \in\Hc_{1}$
such that
\[
1 + \bigl(H^{\lambda,n} \cdot S\bigr)_t \ge\exp \biggl(
\lambda(F\cdot S)_{\tau_a
\wedge t} - \frac{\lambda^2}{2} \bigl(N^{(\rho^n)}_t+1
\bigr) 2^{-2n} a^2 d \biggr) =: \mathcal{E}^{\lambda,n}_{\tau_a \wedge t}
\]
for all $t \in[0,T]$, where
\[
N^{(\rho^n)}_t := \max\bigl\{k : \rho^n_k
\le t\bigr\} \le N^n_t + N^{(\tau)}_t
:= N^n_t + \max\{k : \tau_k \le t\}.
\]
By Lemma~\ref{lem:number}, we have $N^n_t \le2^{2n} V^n_t$, so that
\[
\mathcal{E}^{\lambda,n}_{\tau_a \wedge t} \ge\exp \biggl( \lambda (F\cdot
S)_t - \frac{\lambda^2}{2} V^n_T
a^2 d - \frac{\lambda^2}{2} \bigl( N^{(\tau)}_T +
1\bigr) 2^{-2n} a^2 d \biggr).
\]
If now $\Vert(F\cdot S)\Vert_\infty\ge a b \sqrt{c}$, $\Vert
F(\omega)\Vert_\infty\le a$ and $\langle S \rangle_T \le c$, then
\[
\liminf_{n \rightarrow\infty} \sup_{t \in[0,T]} \frac{\mathcal
{E}^{\lambda,n}_t + \mathcal{E}^{-\lambda,n}_t}{2}
\ge\frac{1}{2} \exp \biggl( \lambda a b \sqrt{c} - \frac{\lambda^2}{2} c
a^2 d \biggr).
\]
The argument inside the exponential is maximized for $\lambda= b/(a
\sqrt{c} d)$, in which case we obtain $1/2 \exp(b^2/(2d))$. The
statement now follows from Remark~\ref{r:vovk take supremum in content}.
\end{pf}

Of course, we did not actually establish an isometry but only an upper
bound for the integral. But this estimate is the key ingredient which
allows us to extend the model free It\^o integral to more general
integrands, and it is this analogy to the classical setting that the
terminology ``model free version of It\^o's isometry'' alludes to.

Let us extend the topology of Section~\ref{sec:topology} to processes:
we identify $X,Y\colon\Omega\times[0,T]\to\R^m$ if for typical
price paths we have $X_t = Y_t$ for all $t \in[0,T]$, and we write
$\overline L_0([0,T],\R^m)$ for the resulting space of equivalence
classes which we equip with the distance
\[
d_\infty(X,Y) := \overline E\bigl[\|X-Y\|_\infty\wedge1\bigr].
\]

Ideally, we would like the stochastic integral on step functions to be
continuous with respect to $d_\infty$. However, using Proposition~\ref
{prop:local martingale} it is easy to see that $\overline P(\|
((1/n)\cdot S)\|_\infty> \varepsilon) = 1$ for all $n \in\N$ and
$\varepsilon>0$. This is why we also introduce for $c>0$ the pseudometric
\[
d_c(X,Y) := \overline E\bigl[\bigl(\|X-Y\|_\infty\wedge1\bigr)
\1_{\langle S \rangle
_T \le c}\bigr] \le d_\infty(X,Y),
\]
and then
\[
d_{\mathrm{loc}}(X,Y) := \sum_{n=1}^\infty2^{-n}
d_{2^n}(X,Y) \le d_\infty(X,Y).
\]
The distance $d_{\mathrm{loc}}$ is somewhat analogous to the distance
used to metrize the topology of uniform convergence on compacts, except
that we do not localize in time but instead we control the size of the
quadratic variation.
For step functions $F$ and $G$, we get from Lemma~\ref{lem:model free
ito isometry}
\begin{eqnarray*}
d_c\bigl((F\cdot S), (G \cdot S)\bigr) & \le&\overline P \bigl(
\bigl\{\bigl\| \bigl((F - G)\cdot S\bigr)\bigr\|_\infty\ge a b \sqrt {c}\bigr\} \cap
\bigl\{\|F - G\|_\infty\le a\bigr\} \cap\bigl\{\langle S \rangle_T \le
c\bigr\} \bigr)
\\
&&{} + \frac{d_c(F,G)}{a} + a b \sqrt{c}
\\
& \le&2 \exp \biggl(-\frac{b^2}{2d} \biggr) + \frac{d_c(F,G)}{a} + a b \sqrt{c}
\end{eqnarray*}
whenever $a,b>0$. Setting $a := \sqrt{d_c(F,G)}$ and $b: = \sqrt{d
|\log a|}$, we deduce that
%
\begin{equation}
\label{eq:ito continuity} d_c\bigl((F\cdot S), (G \cdot S)\bigr) \lesssim(1 +
\sqrt c) d_c(F,G)^{1/2-\varepsilon}
\end{equation}
for all $\varepsilon>0$, and in particular
\[
d_{\mathrm{loc}}\bigl((F\cdot S), (G \cdot S)\bigr) \lesssim\sum
_{n=1}^\infty 2^{-n/2} d_{2^n}(F,G)^{1/2-\varepsilon}
\lesssim d_\infty (F,G)^{1/2-\varepsilon}.
\]

\begin{thmm}\label{thmm:int}
Let $F$ be an adapted, c\`adl\`ag process with values in $\mathbb
{R}^d$. Then there exists $\int F \,\dd S \in\overline L_0([0,T],\R)$
such that for every sequence of step functions $(F^n)$ satisfying $\lim_n d_\infty(F^n, F) = 0$ we have $\lim_n d_{\mathrm{loc}}((F^n \cdot
S), \int F \,\dd S) = 0$. The integral process $\int F \,\dd S$ is
continuous for typical price paths, and there exists a representative
$\int F \,\dd S$ which is adapted, although it may take the values $\pm
\infty$.
We usually write $\int_0^t F_s \,\dd S_s := \int F \,\dd S(t)$, and we
call $\int F \,\dd S$ the \emph{model free It\^o integral of $F$ with
respect to $S$}.

The map $F \mapsto\int F \,\dd S$ is linear, satisfies
\[
d_{\mathrm{loc}} \biggl( \int F \,\dd S, \int G \,\dd S \biggr) \lesssim
d_\infty (F,G)^{1/2-\varepsilon}
\]
for all $\varepsilon> 0$, and the model free version of It\^o's
isometry extends to this setting:
\[
\overline P \biggl( \biggl\{ \biggl\|\int F \,\dd S\biggr\|_\infty\ge a b \sqrt c
\biggr\} \cap\bigl\{\|F \|_\infty\le a \bigr\} \cap\bigl\{\langle S
\rangle_T \le c \bigr\} \biggr) \le2 \exp\bigl(-b^2/(2d)
\bigr)
\]
for all $a,b,c>0$.
\end{thmm}

\begin{pf}
Everything follows in a straightforward way from~\eqref{eq:ito
continuity} in combination with Lemma~\ref{lem:complete metric}. We
have to use the fact that $F$ is adapted and c\`adl\`ag in order to
approximate it uniformly by step functions.
\end{pf}

Another simple consequence of our model free version of It\^o's
isometry is a strengthened version of Karandikar's \cite
{Karandikar1995} pathwise It\^o integral which works for all typical
price paths and not just quasi surely under the local martingale measures.

\begin{cor}\label{cor:pw int}
In the setting of Theorem~\ref{thmm:int}, let $(F^m)_{m \in\N}$ be a
sequence of step functions with $\Vert F^m(\omega) - F(\omega)\Vert
_\infty\le c_m $ for all $\omega\in\Omega$ and all $m \in\N$. Then
for typical price paths $\omega$ there exists a constant $C(\omega) >
0$ such that
%
\begin{equation}
\label{eq:ito integral rate} \biggl\Vert\bigl(F^m \cdot S\bigr) (\omega) - \int F \,\dd S(
\omega)\biggr \Vert _\infty\le C(\omega) c_m \sqrt{\log m}
\end{equation}
for all $m \in\mathbb{N}$. So, if $c_m = o((\log m)^{-1/2})$, then for
typical price paths $(F^m \cdot S)$ converges to $\int F \,\dd S$.
\end{cor}

\begin{pf}
For $c>0$ the model free It\^o isometry gives
\[
\oP \biggl( \biggl\{\biggl\Vert\bigl(F^m\cdot S\bigr) - \int F \,\dd S
\biggr\Vert_\infty\ge c_m \sqrt{4d\log m} \sqrt{c} \biggr\} \cap
\bigl\{\langle S \rangle_T \le c\bigr\} \biggr) \le
\frac{1}{m^2}.
\]
Since this is summable in $m$, the claim follows from Borel Cantelli
(which only requires countable subadditivity and can thus be applied
for the outer measure $\oP$).
\end{pf}

\begin{rmk}
The speed of convergence~\eqref{eq:ito integral rate} is better than
the one that can be obtained using the arguments in~\cite
{Karandikar1995}, where the summability of $(c_m)$ is needed.
\end{rmk}

\begin{rmk}
It would be desirable to extend the robust It\^o integral obtained in
Theorem~\ref{thmm:int} to general locally square integrable integrands,
that is adapted processes $H$ with measurable trajectories and such
that $\int_0^t H^2_s(\omega) \,\dd\langle S \rangle_s (\omega) < \infty$
for all $t$ and for all $\omega$ which have a continuous quadratic
variation $\langle S \rangle(\omega)$ up to time $t$. The reason why
our methods break down in this setting is that our ``model free version
of It\^o's isometry'' requires as input a uniform bound on the
integrand. However, even with the restriction to c\`adl\`ag integrands
our robust It\^o integral is suitable for all (financial) applications
which use Karandikar's pathwise stochastic integral~\cite
{Karandikar1995}, with the great advantage of being a ``model free''
and not just a ``quasi sure'' object.

Similarly, it would be nice to have an extension of Theorem~\ref
{thmm:int} to c\`adl\`ag integrators. Unfortunately, neither the outer
measure $\overline{P}$ nor Vovk's outer measure $\overline{Q}$ have an
obvious reasonable extension to the space $D([0,T],\mathbb{R}^d)$ of
all c\`adl\`ag functions. The problem is that on this space there are
no non-zero admissible strategies. As initiated in \cite{Vovk2011}, it
is possible to consider $\overline{P}$ or $\overline{Q}$ on the
subspace of all paths in $D([0,T],\mathbb{R}^d)$ whose jump size at
time $t>0$ is bounded by a function of their supremum up to time $t$.
However, it would be necessary to develop new techniques to obtain
Theorem~\ref{thmm:int} in this setting since, for instance, the pathwise
Hoeffding inequality (Lemma~\ref{l:hoeffding}) would not be applicable anymore.
\end{rmk}

\section{Rough path integration for typical price paths}\label
{sec:rough paths}

Our second approach to model free stochastic integration is based on
the rough path integral, which has the advantage of being a continuous
linear operator between Banach spaces. The disadvantage is that we have
to restrict the set of integrands to those ``locally looking like
$S$'', modulo a smoother remainder. Our two main results in this
section are that every typical price path has a naturally associated
It\^o rough path, and that the rough path integral can be constructed
as limit of Riemann sums.

Let us start by recalling the basic definitions and results of rough
path theory.

\subsection{The Lyons--Gubinelli rough path integral}

Here we follow more or less the lecture notes~\cite{Friz2014}, to which
we refer for a gentle introduction to rough paths. More advanced
monographs are~\cite{Lyons2002,Lyons2007,Friz2010}. The main
difference to~\cite{Friz2014} in the derivation below is that we use
$p$-variation to describe the regularity, and not H\"older continuity,
because it is not true that all typical price paths are H\"older
continuous. Also, we make an effort to give reasonably sharp results,
whereas in~\cite{Friz2014} the focus lies more on the pedagogical
presentation of the material. We stress that in this subsection we are
merely collecting classical results.

\begin{definition}
A \emph{control function} is a continuous map $c\colon\Delta_T \to
[0,\infty)$ with $c(t,t) = 0$ for all $t \in[0,T]$ and such that
$c(s,u) + c(u,t) \le c(s,t)$ for all $0 \le s \le u \le t \le T$.
\end{definition}

Observe that if $f\colon[0,T] \too\R^d$ satisfies $|f_{s,t}|^p \le
c(s,t)$ for all $(s,t) \in\Delta_T$, then the $p$-variation of $f$ is
bounded from above by $c(0,T)^{1/p}$.

\begin{definition}
Let $p \in(2,3)$. A $p$-rough path is a map $\BS= (S, A) \colon
\Delta_T \to\R^d \times\R^{d\times d}$ such that \emph{Chen's relation}
\[
S^i(s,t) = S^i(s,u) + S^i(u,t) \quad\mbox{and}\quad
A^{i,j}(s,t) = A^{i,j}(s,u) + A^{i,j}(u,t) +
S^i(s,u) S^j(u,t)
\]
holds for all $1 \le i,j \le d$ and $0 \le s \le u \le t \le T$ and
such that there exists a control function $c$ with
\[
\bigl|S(s,t)\bigr|^p + \bigl|A(s,t)\bigr|^{p/2} \le c(s,t)
\]
(in other words $S$ has finite $p$-variation and $A$ has finite
$p/2$-variation). In that case, we call $A$ the \emph{area} of $S$.
\end{definition}

\begin{rmk}
Chen's relation simply states that $S$ is the increment of a function,
that is $S(s,t)= S(0,t)-S(0,s)=S_{s,t}$ for $S_t := S(0,t)$, and that
for all $i,j$ there exists a function $f^{i,j}\colon[0,T] \to\R$ such
that $A^{i,j}(s,t) = f^{i,j}(t) - f^{i,j}(s) - S^i_s S^j_{s,t}$.
Indeed, it suffices to set $f^{i,j}(t) := A^{i,j}(0,t) + S^i_0 S^j_{0,t}$.
\end{rmk}

\begin{rmk}
The (strictly speaking incorrect) name ``area'' stems from the fact
that if $S\colon[0,T] \to\mathbb{R}^2$ is a two-dimensional smooth
function and if
\[
A^{i,j}(s,t) = \int_s^t \int
_s^{r_2} \,\dd S^i_{r_1} \,\dd
S^j_{r_2} = \int_s^t
S^i_{s,r_2} \,\dd S^j_{r_2},
\]
then the antisymmetric part of $A(s,t)$ corresponds to the algebraic
area enclosed by the curve $(S_r)_{r \in[s,t]}$. It is a deep insight
of Lyons~\cite{Lyons1998}, proving a conjecture of F\"ollmer, that the
area is exactly the additional information which is needed to solve
differential equations driven by $S$ in a pathwise continuous manner,
and to construct stochastic integrals as continuous maps.
Actually,~\cite{Lyons1998} solves a much more general problem and
proves that if the driving signal is of finite $p$-variation for some
$p>1$, then it has to be equipped with the iterated integrals up to
order $\lfloor p \rfloor- 1$ to obtain a continuous integral map. The
for us relevant case $p \in(2,3)$ was already treated in~\cite{Lyons1995a}.
\end{rmk}

\begin{ex}
If $S$ is a continuous semimartingale and if we set $S(s,t) :=
S_{s,t}$ as well as
\[
A^{i,j}(s,t) := \int_s^t \int
_s^{r_2} \,\dd S^i_{r_1} \,\dd
S^j_{r_2} = \int_s^t
S^i_{s,r_2} \,\dd S^j_{r_2},
\]
where the integral can be understood either in the It\^o or in the
Stratonovich sense, then almost surely $\BS= (S,A)$ is a $p$-rough
path for all $p \in(2,3)$. This is shown in \cite{Coutin2005}, and we
will give a simplified model free proof below (indeed we will show that
every typical price path together with its model free It\^o integral is
a $p$-rough path for all $p \in(2,3)$, from where the statement about
continuous semimartingales easily follows).
\end{ex}

From now on, we fix $p \in(2,3)$ and we assume that $\BS$ is a
$p$-rough path. Gubinelli~\cite{Gubinelli2004} observed that for every
rough path there is a naturally associated Banach space of integrands,
the space of \emph{controlled paths}. Heuristically, a path $F$ is
controlled by $S$, if it locally ``looks like $S$'', modulo a smooth
remainder. The precise definition is the following.
%
\begin{definition}
Let $p \in(2,3)$ and $q>0$ be such that $2/p + 1/q > 1$. Let $\mathbb
{S} = (S,A)$ be a $p$-rough path and let $F\colon[0,T] \too\R^n$ and
$F'\colon[0,T] \too\R^{n\times d}$. We say that the pair $(F,F')$ is
\emph{controlled} by $S$ if the \emph{derivative} $F'$ has finite
$q$-variation, and the \emph{remainder} $R_F \colon\Delta_T \too\R
^n$, defined by
\[
R_F(s,t) := F_{s,t} - F'_s
S_{s,t},
\]
has finite $r$-variation for $1/r = 1/p + 1/q$. In this case, we write
$(F,F') \in\CC_\BS^{q} = \CC_\BS^{q}(\R^n)$, and define
\[
\bigl\|\bigl(F, F'\bigr)\bigr\|_{\CC_\BS^{q}} := \bigl\| F'
\bigr\|_{q\mbox{-}\operatorname{var}} + \| R_F \|_{r\mbox
{-}\operatorname{var}}.
\]
Equipped with the norm $|F_0| + |F'_0| + \|(F,F')\|_{\CC_\BS^{q}}$, the
space $\CC_\BS^{q}$ is a Banach space.
\end{definition}

Naturally, the function $F'$ should be interpreted as the derivative of
$F$ with respect to $S$. The reason for considering couples $(F,F')$
and not just functions $F$ is that the regularity requirement on the
remainder $R_F$ usually does not determine $F'$ uniquely for a given
path $F$. For example, if $F$ and $S$ both have finite $r$-variation
rather than just finite $p$-variation, then for every $F'$ of finite
$q$-variation we have $(F,F') \in\CC_\BS^q$.

Note that we do not require $F$ or $F'$ to be continuous. We will point
out in Remark~\ref{rmk:non-continuous integrands} below why this does
not pose any problem.

To gain a more ``quantitative'' feeling for the condition on $q$, let
us assume for the moment that we can choose $p>2$ arbitrarily close to
2 (which is the case in the example of a continuous semimartingale
rough path). Then $2/p + 1/q > 1$ as long as $q > 0$, so that the
derivative $F'$ may essentially be as irregular as we want. The
remainder $R_F$ has to be of finite $r$-variation for $1/r = 1/p +
1/q$, so in other words it should be of finite $r$-variation for some
$r <2$ and thus slightly more regular than the sample path of a
continuous local martingale.

\begin{ex}
Let $\varepsilon\in(0,1]$ be such that $(2+\varepsilon)/p>1$. Let
$\varphi\in C^{1+\varepsilon}_b$ and define $F_s := \varphi(S_s)$ and
$F'_s := \varphi'(S_s)$. Then $(F, F') \in\CC_\BS^{p/\varepsilon}$:
Clearly $F'$ has finite $p/\varepsilon$-variation. For the remainder,
we have
\[
\bigl|R_F(s,t)\bigr|^{p/(1+\varepsilon)} = \bigl| \varphi(S_t) -
\varphi(S_s) - \varphi '(S_s)
S_{s,t}\bigr|^{p/(1+\varepsilon)} \le \| \varphi\|_{C^{1+\varepsilon
}_b} c(s,t),
\]
where $c$ is a control function for $S$.
As the image of the continuous path $S$ is compact, it is not actually
necessary to assume that $\varphi$ is bounded. We may always consider a
$C^{1+\varepsilon}$ function $\psi$ of compact support, such that $\psi
$ agrees with $\varphi$ on the image of $S$.

This example shows that in general $R_F(s,t)$ is not a path increment
of the form $R_F(s,t) = G(t) - G(s)$ for some function $G$ defined on
$[0,T]$, but really a function of two variables.
\end{ex}

\begin{ex}
Let $G$ be a path of finite $r$-variation for some $r$ with $1/p + 1/r
> 1$. Setting $(F,F') = (G,0)$, we obtain a controlled path in $\CC
^q_\BS$, where $1/q = 1/r - 1/p$. In combination with Theorem~\ref
{thmm:rough path integral} below, this example shows in particular that
the controlled rough path integral extends the Young integral and the
Riemann--Stieltjes integral.
\end{ex}

The basic idea of rough path integration is that if we already know how
to define $\int S \,\dd S$, and if $F$ looks like $S$ on small scales,
then we should be able to define $\int F \,\dd S$ as well. The precise
result is given by the following theorem.

\begin{thmm}[(Theorem~4.9 in~\cite{Friz2014}, see also~\cite
{Gubinelli2004}, Theorem~1)]\label{thmm:rough path integral}
Let $p \in(2,3)$ and $q>0$ be such that $2/p + 1/q > 1$. Let $\mathbb
{S} = (S,A)$ be a $p$-rough path and let $(F,F') \in\CC_\BS^{q}$. Then
there exists a unique function $\int F \,\dd S \in C([0,T], \R^n)$ which satisfies
\[
\biggl| \int_s^t F_u \,\dd
S_u - F_s S_{s,t} - F'_s
A(s,t) \biggr| \lesssim\| S \|_{p\mbox{-}\var,[s,t]} \| R_F\|_{r\mbox{-}\var,[s,t]} +
\| A \| _{p/2\mbox{-}\var,[s,t]} \bigl\|F'\bigr\|_{q\mbox{-}\var,[s,t]}
\]
for all $(s,t) \in\Delta_T$. The integral is given as limit of the
compensated Riemann sums
%
\begin{equation}
\label{eq:compensated riemann sums} \int_0^t F_u
\,\dd S_u = \lim_{m \too\infty} \sum
_{[s_1,s_2] \in\pi^m} \bigl[ F_{s_1} S_{s_1,s_2} +
F'_{s_1} A(s_1,s_2) \bigr],
\end{equation}
where $(\pi^m)$ is any sequence of partitions of $[0,t]$ with mesh size
going to $0$.

The map $(F,F') \mapsto(G,G') := (\int F_u \,\dd S_u, F)$ is continuous
from $\CC_\BS^{q}$ to $\CC_\BS^{p}$ and satisfies
\[
\bigl\| \bigl(G,G'\bigr)\bigr\|_{\CC_\BS^{p}} \lesssim\|F
\|_{p\mbox{-}\var} + \bigl(\bigl\|F'\bigr\| _\infty+
\bigl\|F'\bigr\|_{q\mbox{-}\var}\bigr) \|A\|_{p/2\mbox{-}\var} + \|S\|
_{p\mbox{-}\var} \|R_F\|_{r\mbox{-}\var}.
\]
\end{thmm}

\begin{rmk}\label{rmk:non-continuous integrands}
To the best of our knowledge, there is no publication in which the
controlled path approach to rough paths is formulated using
$p$-variation regularity. The references on the subject all work with
H\"older continuity. But in the $p$-variation setting, all the proofs
work exactly as in the H\"older setting, and it is a simple exercise to
translate the proof of Theorem~4.9 in~\cite{Friz2014} (which is based
on Young's maximal inequality which we will encounter below) to obtain
Theorem~\ref{thmm:rough path integral}.

There is only one small pitfall: We did not require $F$ or $F'$ to be
continuous. The rough path integral for discontinuous functions is
somewhat tricky, see~\cite{Williams2001,Friz2012Levy}. But here we do
not run into any problems, because the integrand $\BS= (S,A)$ is
continuous. The construction based on Young's maximal inequality works
as long as integrand and integrator have no common discontinuities, see
the theorem on page~264 of~\cite{Young1936}.
\end{rmk}

If now $\varphi\in C^{1+\varepsilon}_b$ for some $\varepsilon> 0$,
then using a Taylor expansion one can show that there exist $p > 2$ and
$q>0$ with $2/p + 1/q > 0$, such that $(F,F') \mapsto(\varphi(F),
\varphi'(F) F')$ is a locally bounded map from $\CC_\BS^{p}$ to $\CC_\BS
^{q}$. Combining this with the fact that the rough path integral is a
bounded map from $\CC_\BS^{q}$ to $\CC_\BS^{p}$, it is not hard to
prove the \emph{existence} of solutions to the rough differential equation
%
\begin{equation}
\label{eq:rde} X_t = x_0+ \int_0^t
\varphi(X_s)\,\dd S_s,
\end{equation}
$t \in[0,T]$, where $X \in\CC_\BS^{p}$, $\int\varphi(X_s) \,\dd S_s$
denotes the rough path integral, and $S$ is a typical price path.
Similarly, if $\varphi\in C^{2+\varepsilon}_b$, then the map $(F,F')
\mapsto(\varphi(F), \varphi'(F) F')$ is locally Lipschitz continuous
from $\CC_\BS^{p}$ to $\CC_\BS^{q}$, and this yields the \emph
{uniqueness} of the solution to~\eqref{eq:rde} -- at least among the
functions in the Banach space $\CC_\BS^{p}$. See Section~5.3 of~\cite
{Gubinelli2004} for details.

A remark is in order about the stringent regularity requirements on
$\varphi$. In the classical It\^o theory of SDEs, the function $\varphi
$ is only required to be Lipschitz continuous. But to solve a
Stratonovich SDE, we need better regularity of $\varphi$. This is
natural, because the Stratonovich SDE can be rewritten as an It\^o SDE
with a Stratonovich correction term: the equations
\begin{eqnarray*}
\dd X_t& =& \varphi(X_t) \circ\,\dd W_t\quad
\mbox{and}
\\
\dd X_t &=& \varphi(X_t) \,\dd W_t +
\tfrac{1}{2} \varphi'(X_t) \varphi
(X_t) \,\dd t
\end{eqnarray*}
are equivalent (where $W$ is a standard Brownian motion, $\dd W_t$
denotes It\^o integration, and $\circ\,\dd W_t$ denotes Stratonovich
integration). To solve the second equation, we need $\varphi' \varphi$
to be Lipschitz continuous, which is always satisfied if $\varphi\in
C^2_b$. But rough path theory cannot distinguish between It\^o and
Stratonovich integrals: If we define the area of $W$ using It\^o
(resp., Stratonovich) integration, then the rough path solution
of the equation will coincide with the It\^o (resp.,
Stratonovich) solution. So in the rough path setting, the function
$\varphi$ should satisfy at least the same conditions as in the
Stratonovich setting. The regularity requirements on $\varphi$ are
essentially sharp, see~\cite{Davie2007}, but the boundedness assumption
can be relaxed, see~\cite{Lejay2012}. See also Section~10.5 of \cite
{Friz2010} for a slight relaxation of the regularity requirements in
the Brownian case.

Of course, the most interesting result of rough path theory is that the
solution to a rough differential equation depends continuously on the
driving signal. This is a consequence of the following observation.

\begin{prop}[(Proposition~9.1 of~\cite{Friz2014})]\label{prop:continuous
rough path integral}
Let $p \in(2,3)$ and $q>0$ with $2/p + 1/q > 0$. Let $\BS= (S,A)$
and $\tilde{\BS} = (\tilde{S}, \tilde{A})$ be two $p$-rough paths, let
$(F, F') \in\CC_\BS^{q}$ and $(\tilde{F}, \tilde{F}') \in\CC_{\tilde
{\BS}}^{q}$. 
Then for every $M > 0$ there exists $C_M > 0$ such that
\begin{eqnarray*}
&&\biggl\| \int_0^\cdot F_s \,\dd
S_s - \int_0^\cdot
\tilde{F}_s \,\dd\tilde {S}_s \biggr\|_{p\mbox{-}\var}\\
&&\quad \le
C_M \bigl( |F_0 - \tilde{F}_0| +
\bigl|F'_0 - \tilde{F}'_0\bigr| +
\bigl\|F' - \tilde{F}'\bigr \|_{q\mbox{-}\var}
 \\
 &&\qquad{}+ \| R_F - R_{\tilde{F}}\|_{r\mbox{-}\var} + \| S -
\tilde{S}\|_{p\mbox{-}\var} + \| A - \tilde{A}\|_{p/2\mbox{-}\var} \bigr),
\end{eqnarray*}
as long as
\begin{eqnarray*}
&&\max\bigl\{ \bigl|F'_0\bigr| + \bigl\|\bigl(F,F'
\bigr)\bigr\|_{\CC_\BS^{q}}, \bigl|\tilde{F}'_0\bigr| +\bigl \|\bigl(
\tilde {F},\tilde{F}'\bigr)\bigr\|_{\CC_{\tilde{\BS}}^{q}}, \|S
\|_{p\mbox{-}\var}, \|A\| _{p/2\mbox{-}\var}, \|\tilde{S}\|_{p\mbox{-}\var}, \|
\tilde{A}\| _{p/2\mbox{-}\var}\bigr\}\\
&&\quad \le M.
\end{eqnarray*}
\end{prop}

In other words, the rough path integral depends on integrand and
integrator in a locally Lipschitz continuous way, and therefore it is
no surprise that the solutions to differential equations driven by
rough paths depend continuously on the signal.

\subsection{Typical price paths as rough paths}

Our second approach to stochastic integration in model free financial
mathematics is based on the rough path integral. Here we show that for
every typical price path, the pair $(S, A)$ is a $p$-rough path for all
$p\in(2,3)$, where $A$ corresponds to the model free It\^o integral
$\int S \,\dd S$ which we constructed in Section~\ref{sec:cpi}. We also
show that many Riemann sum approximations to $\int S \,\dd S$ uniformly
satisfy a certain coarse grained regularity condition, which we will
use in the following section to prove that in our setting rough path
integrals can be calculated as limit of Riemann sums (and not
compensated Riemann sums as in Theorem~\ref{thmm:rough path integral}).
The main ingredient in the proofs will be our speed of
convergence~\eqref{eq:ito integral rate}.

\begin{thmm}\label{thmm:area variation}
For $(s,t) \in\Delta_T$, $\omega\in\Omega$, and $i,j \in\{1, \ldots
, d\}$ define
\[
A^{i,j}_{s,t}(\omega) := \int_s^t
S^i_r \,\dd S^j_r(\omega) -
S^i_s(\omega ) S^j_{s,t}(\omega)
:= \int_0^t S^i_r
\,\dd S^j_r (\omega) - \int_0^s
S^i_r \,\dd S^j_r(\omega) -
S^i_s(\omega) S^j_{s,t}(\omega),
\]
where $\int S^i \,\dd S^j$ is the integral constructed in Theorem~\ref
{thmm:int}. If $p>2$, then for typical price paths $A = (A^{i,j})_{1\le
i,j \le d}$ has finite $p/2$-variation, and in particular $\BS= (S,A)$
is a $p$-rough path.
\end{thmm}

\begin{pf}
Define the dyadic stopping times $(\tau^n_k)_{n,k \in\N}$ by $\tau
^n_0 := 0$ and
\[
\tau^n_{k+1} := \inf\bigl\{ t \ge\tau^n_k:
|S_t - S_{\tau^n_k}| = 2^{-n}\bigr\},
\]
and set $S^n_t := \sum_k S_{\tau^n_k} \1_{[\tau^n_k, \tau
^n_{k+1})}(t)$, so that $\|S^n - S\|_\infty\le2^{-n}$. Accorcing
to~\eqref{eq:ito integral rate}, for typical price paths $\omega$ there
exists $C(\omega) > 0$ such that
\[
\biggl\|\bigl(S^n \cdot S\bigr) (\omega) - \int S \,\dd S (\omega)
\biggr\|_\infty\le C(\omega) 2^{-n} \sqrt{\log n}.
\]
Fix such a typical price path $\omega$, which is also of finite
$q$-variation for all $q > 2$ (recall from Corollary~\ref{cor:typical
price paths have finite p-variation} that this is satisfied by typical
price paths). Let us show that for such $\omega$, the process $A$ is of
finite $p/2$-variation for all $p > 2$.

We have for $(s,t) \in\Delta_T$, omitting the argument $\omega$ of
the processes under consideration,
\begin{eqnarray*}
|A_{s,t}| & \le&\biggl | \int_s^t
S_r \,\dd S_r - \bigl(S^n \cdot S
\bigr)_{s,t}\biggr | + \bigl| \bigl(S^n \cdot S\bigr)_{s,t} -
S_s S_{s,t}\bigr|
\\
& \le& C 2^{-n} \sqrt{\log n} +\bigl | \bigl(S^n \cdot S
\bigr)_{s,t} - S_s S_{s,t}\bigr| \lesssim_\varepsilon
C 2^{-n(1-\varepsilon)} + \bigl| \bigl(S^n \cdot S\bigr)_{s,t} -
S_s S_{s,t}\bigr|
\end{eqnarray*}
for every $n \in\N$, $\varepsilon>0$. The second term on the
right-hand side can be estimated, using an argument based on Young's
maximal inequality (see~\cite{Lyons2007}, Theorem~1.16), by
%
\begin{equation}
\label{eq:area variation pr1} \bigl| \bigl(S^n \cdot S\bigr)_{s,t} -
S_s S_{s,t}\bigr| \lesssim\max\bigl\{2^{-n}
c(s,t)^{1/q}, \bigl(\#\bigl\{ k: \tau^n_k
\in[s,t] \bigr\}\bigr)^{1-2/q} c(s,t)^{2/q} + c(s,t)^{2/q}
\bigr\},
\end{equation}
where $c(s,t)$ is a control function with $|S_{s,t}|^q \le c(s,t)$ for
all $(s,t) \in\Delta_T$.
Indeed, if there exists no $k$ with $\tau^n_k \in[s,t]$, then $| (S^n
\cdot S)_{s,t} - S_s S_{s,t}| \le2^{-n} c(s,t)^{1/q}$, using that
$|S_{s,t}| \le c(s,t)^{1/q}$. This corresponds to the first term in the
maximum in~\eqref{eq:area variation pr1}.

Otherwise, note that at the price of adding $c(s,t)^{2/q}$ to the
right-hand side, we may suppose that $s = \tau^n_{k_0}$ for some $k_0$.
Let now $\tau^n_{k_0}, \ldots, \tau^n_{k_0 + N-1}$ be those $(\tau
^n_k)_k$ which are in $[s,t)$. Without loss of generality we may
suppose $N \ge2$, because otherwise $(S^n \cdot S)_{s,t} = S_s
S_{s,t}$. Abusing notation, we write $\tau^n_{k_0 + N} = t$. The idea
is now to successively delete points $(\tau^n_{k_0 + \ell})$ from the
partition, in order to pass from $(S^n \cdot S)$ to $S_s S_{s,t}$. By
super-additivity of $c$, there must exist $\ell\in\{1, \ldots, N-1\}$,
for which
\[
c\bigl(\tau^n_{k_0 + \ell-1}, \tau^n_{k_0 + \ell+ 1}
\bigr) \le\frac{2}{N-1} c(s,t).
\]
Deleting $\tau^n_{k_0 + \ell}$ from the partition and subtracting the
resulting integral from $(S^n\cdot S)_{s,t}$, we get
\begin{eqnarray*}
&&|S_{\tau^n_{k_0 + \ell-1}} S_{\tau^n_{k_0 + \ell-1}, \tau^n_{k_0 +
\ell}} + S_{\tau^n_{k_0 + \ell}} S_{\tau^n_{k_0 + \ell}, \tau^n_{k_0 +
\ell+ 1}} -
S_{\tau^n_{k_0 + \ell-1}} S_{\tau^n_{k_0 + \ell-1}, \tau
^n_{k_0 + \ell+ 1}}|
\\
&&\quad= |S_{\tau^n_{k_0 + \ell-1},\tau^n_{k_0 + \ell}} S_{\tau^n_{k_0 +
\ell}, \tau^n_{k_0 + \ell+ 1}}| \le c\bigl(\tau^n_{k_0 + \ell-1},
\tau ^n_{k_0 + \ell+ 1}\bigr)^{2/q} \le \biggl(
\frac{2}{N-1} c(s,t) \biggr)^{2/q}.
\end{eqnarray*}
Successively deleting all the points except $\tau^n_{k_0} = s$ and $\tau
^n_{k_0 + N} = t$ from the partition gives
\[
\bigl| \bigl(S^n \cdot S\bigr)_{s,t} - S_s
S_{s,t} \bigr| \le\sum_{k=2}^{N}
\biggl(\frac
{2}{k-1} c(s,t) \biggr)^{2/q} \lesssim N^{1-2/q}
c(s,t)^{2/q},
\]
and therefore~\eqref{eq:area variation pr1}. Now it is easy to see that
$\#\{ k: \tau^n_k \in[s,t] \} \le2^{nq} c(s,t)$ (compare also the
proof of Lemma~\ref{lem:number}), and thus
%
\begin{eqnarray}
\label{eq:area variation pr2}
|A_{s,t}| & \lesssim_\varepsilon& C
2^{-n(1-\varepsilon)} + \max\bigl\{ 2^{-n} c(s,t)^{1/q},
\bigl(2^{nq} c(s,t)\bigr)^{1-2/q} c(s,t)^{2/q} +
c(s,t)^{2/q}\bigr\}
\nonumber
\\[-8pt]
\\[-8pt]
\nonumber
& = &C 2^{-n(1-\varepsilon)} + \max\bigl\{2^{-n} c(s,t)^{1/q},
2^{-n(2-q)} c(s,t) + c(s,t)^{2/q}\bigr\}.
\end{eqnarray}
This holds for all $n \in\N$, $\varepsilon> 0$, $q > 2$. Let us
suppose for the moment that $c(s,t) \le1$ and let $\alpha> 0$ to be
determined later. Then there exists $n \in\N$ for which $2^{-n-1} <
c(s,t)^{1/\alpha(1-\varepsilon)} \le2^{-n}$. Using this $n$ in~\eqref
{eq:area variation pr2}, we get
\begin{eqnarray*}
&&|A_{s,t}|^\alpha\\
&&\quad \lesssim_{\varepsilon,\omega,\alpha} c(s,t) + \max \bigl
\{c(s,t)^{1/(1-\varepsilon)} c(s,t)^{\alpha/q}, c(s,t)^{(2-q)/(1-\varepsilon) + \alpha} +
c(s,t)^{2\alpha/q} \bigr\}
\\
&&\quad = c(s,t) + \max \bigl\{c(s,t)^{\frac{q+\alpha(1-\varepsilon
)}{q(1-\varepsilon)}}, c(s,t)^{\frac{2-q+\alpha(1-\varepsilon
)}{1-\varepsilon}} +
c(s,t)^{2\alpha/q} \bigr\}.
\end{eqnarray*}
We would like all the exponents in the maximum on the right-hand side
to be larger or equal to~1. For the first term, this is satisfied as
long as $\varepsilon<1$. For the third term, we need $\alpha\ge q/2$.
For the second term, we need $\alpha\ge(q-1-\varepsilon)/(1-\varepsilon
)$. Since $\varepsilon> 0$ can be chosen arbitrarily close to $0$, it
suffices if $\alpha> q-1$. Now, since $q > 2$ can be chosen
arbitrarily close to $2$, we see that $\alpha$ can be chosen
arbitrarily close to 1. In particular, we may take $\alpha= p/2$ for
any $p>2$, and we obtain $|A_{s,t}|^{p/2} \lesssim_{\omega,\delta} c(s,t)$.

It remains to treat the case $c(s,t) > 1$, for which we simply estimate
\[
|A_{s,t}|^{p/2} \lesssim_p \biggl\| \int
_0^\cdot S_r \,\dd S_r \biggr\|
_\infty^{p/2} + \|S\|_\infty^{p} \le \biggl(
\biggl\| \int_0^\cdot S_r \,\dd
S_r \biggr\|_\infty^{p/2} + \|S\|_\infty^{p}
\biggr) c(s,t).
\]
So for every interval $[s,t]$ we can estimate $|A_{s,t}|^{p/2} \lesssim
_{\omega,p} c(s,t)$, and the proof is complete.
\end{pf}

\begin{rmk}
To the best of our knowledge, this is one of the first times that a
non-geometric rough path is constructed in a non-probabilistic setting,
and certainly we are not aware of any works where rough paths are
constructed using financial arguments.

We also point out that, thanks to Proposition~\ref{prop:local
martingale}, we gave a simple, model free, and pathwise proof for the
fact that a local martingale together with its It\^o integral defines a
rough path. While this seems intuitively clear, the only other proof
that we know of is somewhat involved: it relies on a strong version of
the Burkholder--Davis--Gundy inequality, a time change, and Kolmogorov's
continuity criterion; see~\cite{Coutin2005} or Chapter~14 of~\cite{Friz2010}.
\end{rmk}

The following auxiliary result will allow us to obtain the rough path
integral as a limit of Riemann sums, rather than compensated Riemann
sums, which are usually used to define it.

\begin{lem}\label{lem:uniformly bounded variation}
Let $(c_n)_{n \in\N}$ be a sequence of positive numbers such that
$c_n = o((\log n)^{-c})$ for all $c>0$. For $n \in\N$ define $\tau^n_0
:= 0$ and $\tau^n_{k+1} := \inf\{t \ge\tau^n_k: |S_t - S_{\tau^n_k}| =
c_n\}$, $k \in\N$, and set $S^n_t := \sum_k S_{\tau^n_k} \1_{[\tau
^n_k, \tau^n_{k+1})}(t)$. Then for typical price paths, $((S^n \cdot
S))$ converges uniformly to $\int S \,\dd S$ defined in Theorem~\ref
{thmm:int}. Moreover, for $p > 2$ and for typical price paths there
exists a control function $c = c(p,\omega)$ such that
\[
\sup_n \sup_{k < \ell} \frac{| (S^n \cdot S)_{\tau^n_k,\tau^n_\ell
}(\omega) - S_{\tau^n_k}(\omega) S_{\tau^n_k,\tau^n_\ell}(\omega
)|^{p/2}}{c(\tau^n_k,\tau^n_\ell)}
\le1.
\]
\end{lem}

\begin{pf}
The uniform convergence of $((S^n \cdot S))$ to $\int S \,\dd S$ follows
from Corollary~\ref{cor:pw int}. For the second claim, fix $n \in\N$
and $k < \ell$ such that $\tau^n_\ell\le T$. Then
%
\begin{eqnarray}
\label{eq:uniformly bounded variation pr1}
\bigl|\bigl(S^n \cdot S\bigr)_{\tau^n_k,\tau^n_\ell} -
S_{\tau^n_k} S_{\tau^n_k,\tau
^n_\ell} \bigr| & \lesssim& \biggl\| \bigl(S^n \cdot S
\bigr) - \int_0^\cdot S_s \,\dd
S_s \biggr\|_\infty+ | A_{\tau^n_k, \tau^n_\ell} |
\nonumber
\\[-8pt]
\\[-8pt]
\nonumber
& \lesssim_\omega& c_n \sqrt{\log n} + v_{p/2}
\bigl(\tau^n_k, \tau^n_\ell
\bigr)^{2/p} \lesssim_\varepsilon c_n^{1-\varepsilon} +
v_{p/2}\bigl(\tau^n_k, \tau^n_\ell
\bigr)^{2/p},
\end{eqnarray}
where $\varepsilon> 0$ and the last estimate holds by our assumption
on the sequence $(c_n)$, and where $v_{p/2}(s,t) := \| A\|_{p/2\mbox
{-}\var,[s,t]}^{p/2}$ for $(s,t) \in\Delta_T$. Of course, this
inequality only holds for typical price paths and not for all $\omega
\in\Omega$.

On the other side, the same argument as in the proof of Theorem~\ref
{thmm:area variation} (using Young's maximal inequality and
successively deleting points from the partition) shows that
%
\begin{equation}
\label{eq:uniformly bounded variation pr2} \bigl|\bigl(S^n \cdot S\bigr)_{\tau^n_k,\tau^n_\ell} -
S_{\tau^n_k} S_{\tau^n_k,\tau
^n_\ell} \bigr| \lesssim c_n^{2-q}
v_q\bigl(\tau^n_k,\tau^n_\ell
\bigr),
\end{equation}
where $v_q(s,t) := \| S\|_{q\mbox{-}\var,[s,t]}^q$ for $(s,t) \in\Delta_T$.

Let us define the control function $\tilde{c} := v_q +v_{p/2}$. Take
$\alpha> 0$ to be determined below. If $c_n > \tilde{c}(s,t)^{1/\alpha
(1-\varepsilon)}$, then we use~\eqref{eq:uniformly bounded variation
pr2} and the fact that $2 - q < 0$, to obtain
\[
\bigl|\bigl(S^n \cdot S\bigr)_{\tau^n_k,\tau^n_\ell} - S_{\tau^n_k}
S_{\tau^n_k,\tau
^n_\ell} \bigr|^\alpha\lesssim\bigl(\tilde{c}\bigl(
\tau^n_k,\tau^n_\ell\bigr)
\bigr)^{\frac
{2-q}{(1-\varepsilon)}} v_q\bigl(\tau^n_k,
\tau^n_\ell\bigr)^\alpha\le\tilde {c}\bigl(
\tau^n_k,\tau^n_\ell
\bigr)^{\frac{2-q + \alpha(1-\varepsilon
)}{(1-\varepsilon)}}.
\]
The exponent is larger or equal to 1 as long as $\alpha\ge(q - 1 -
\varepsilon)/(1-\varepsilon)$. Since $q$ and $\varepsilon$ can be
chosen arbitrarily close to $2$ and $0$, respectively, we can take
$\alpha= p/2$, and get
\[
\bigl|\bigl(S^n \cdot S\bigr)_{\tau^n_k,\tau^n_\ell} - S_{\tau^n_k}
S_{\tau^n_k,\tau
^n_\ell} \bigr|^{p/2} \lesssim\tilde{c}\bigl(\tau^n_k,
\tau^n_\ell\bigr) \bigl(1 + \tilde {c}(0,T)^\delta
\bigr)
\]
for a suitable $\delta> 0$.

On the other side, if $c_n \le\tilde{c}(s,t)^{1/\alpha(1-\varepsilon
)}$, then we use~\eqref{eq:uniformly bounded variation pr1} to obtain
\[
\bigl|\bigl(S^n \cdot S\bigr)_{\tau^n_k,\tau^n_\ell} - S_{\tau^n_k}
S_{\tau^n_k,\tau
^n_\ell} \bigr|^\alpha\lesssim\tilde{c}\bigl(\tau^n_k,
\tau^n_\ell\bigr) + \tilde {c}\bigl(\tau^n_k,
\tau^n_\ell\bigr)^{2\alpha/p},
\]
so that also in this case we may take $\alpha= p/2$, and thus we have
in both cases
\[
\bigl|\bigl(S^n \cdot S\bigr)_{\tau^n_k,\tau^n_\ell} - S_{\tau^n_k}
S_{\tau^n_k,\tau
^n_\ell} \bigr|^{p/2} \le c\bigl(\tau^n_k,
\tau^n_\ell\bigr),
\]
where $c$ is a suitable ($\omega$-dependent) multiple of $\tilde{c}$.
\end{pf}

\subsection{The rough path integral as limit of Riemann sums}

Theorem~\ref{thmm:area variation} shows that we can apply the
controlled rough path integral in model free financial mathematics
since every typical price path is a rough path. But there remains a
philosophical problem: As we have seen in Theorem~\ref{thmm:rough path
integral}, the rough path integral $\int F \,\dd S$ is given as limit of
the compensated Riemann sums
\[
\int_0^t F_s \,\dd
S_s = \lim_{m \too\infty} \sum
_{[r_1,r_2] \in\pi^m} \bigl[ F_{r_1} S_{r_1,r_2} +
F'_{r_1} A(r_1,r_2) \bigr],
\]
where $(\pi^m)$ is an arbitrary sequence of partitions of $[0,t]$ with
mesh size going to 0. The term $F_{r_1} S_{r_1,r_2}$ has an obvious
financial interpretation as profit made by buying $F_{r_1}$ units of
the traded asset at time $r_1$ and by selling them at time $r_2$.
However, for the ``compensator'' $F'_{r_1} A(r_1,r_2)$ there seems to
be no financial interpretation, and therefore it is not clear whether
the rough path integral can be understood as profit obtained by
investing in $S$.

However, we observed in Section~\ref{sec:cpi} that along suitable
stopping times $(\tau^n_k)_{n,k}$, we have
\[
\int_0^t S_s \,\dd
S_s = \lim_{n \too\infty} \sum
_k S_{\tau^n_k} S_{\tau
^n_k \wedge t,\tau^n_{k+1} \wedge t}.
\]
By the philosophy of controlled paths, we expect that also for $F$
which looks like $S$ on small scales we should obtain
\[
\int_0^t F_s \,\dd
S_s = \lim_{n \too\infty} \sum
_k F_{\tau^n_k} S_{\tau
^n_k \wedge t,\tau^n_{k+1} \wedge t},
\]
without having to introduce the compensator $F'_{\tau^n_k} A(\tau^n_k
\wedge t, \tau^n_{k+1} \wedge t)$ in the Riemann sum. With the results
we have at hand, this statement is actually relatively easy to prove.
Nonetheless, it seems not to have been observed before.

For the remainder of this section, we fix $S \in C([0,T],\R^d)$, and we
work under the following assumption:

\renewcommand{\theass}{\textup{(}\textsc{Rie}\textup{)}}
\begin{ass}\label{assrie}
Let $\pi^n = \{0 = t^n_0 < t^n_1 < \cdots< t^n_{N_n} = T\}$, $n \in\N
$, be a given sequence of partitions such that $\sup\{ |
S_{t^n_k,t^n_{k+1}}|: k=0, \ldots, N_n-1\}$ converges to 0, and let $p
\in(2,3)$. Set
\[
S^n_t := \sum_{k=0}^{N_n - 1}
S_{t^n_k} \1_{[t^n_k, t^n_{k+1})}(t).
\]
We assume that the Riemann sums $(S^n \cdot S)$ converge uniformly to a
limit that we denote by $\int S \,\dd S$, and that there exists a control
function $c$ for which
%
\begin{equation}
\label{eq:bounded p/2 variation assumption} \sup_{(s,t) \in\Delta_T} \frac{|S_{s,t}|^p}{c(s,t)} + \sup
_n \sup_{0
\le k < \ell\le N_n} \frac{| (S^n \cdot S)_{t^n_k,t^n_\ell} -
S_{t^n_k} S_{t^n_k,t^n_\ell}|^{p/2}}{c(t^n_k,t^n_\ell)} \le1.
\end{equation}
\end{ass}

\begin{rmk}
We expect that ``coarse-grained'' regularity conditions as in~\eqref
{eq:bounded p/2 variation assumption} have been used for a long time,
but were only able to find quite recent references: condition~\eqref
{eq:bounded p/2 variation assumption} was previously used in~\cite
{Perkowski2014a}, see also \cite{Gubinelli2014Schauder}, and has also
appeared independently in~\cite{Kelly2014}. In our setting this is
quite a natural relaxation of a uniform $p$-variation bound since say
for $s,t \in[t^n_k,t^n_{k+1}]$ with $|t-s| \ll|t^n_{k+1} - t^n_k|$
the increment of the discrete integral $(S^n\cdot S)_{s,t}$ is not a
good approximation of $\int_s^t S_r \,\dd S_r$, and therefore we cannot
expect it to be close to $S_s S_{s,t}$.
\end{rmk}

\begin{rmk}
Every typical price path satisfies \ref{assrie} if we choose
$(t^n_k)$ to be a partition of stopping times such as the $(\tau^n_k)$
in Lemma~\ref{lem:uniformly bounded variation}.
\end{rmk}

It is not hard to see that if $S$ satisfies \ref{assrie} and if we
define $A(s,t) := \int_s^t S_r \,\dd S_r - S_s S_{s,t}$, then $(S,A)$ is
a $p$-rough path. This means that we can calculate the rough path
integral $\int F \,\dd S$ whenever $(F,F')$ is controlled by $S$, and the
aim of the remainder of this section is to show that this integral is
given as limit of (uncompensated) Riemann sums. Our proof is somewhat
indirect. We translate everything from It\^o type integrals to related
Stratonovich type integrals, for which the convergence follows from the
continuity of the rough path integral, Proposition~\ref{prop:continuous
rough path integral}. Then we translate everything back to our It\^o
type integrals. To go from It\^o to Stratonovich, we need the quadratic
variation.

\begin{lem}\label{lem:quadratic variation}
Under Assumption~\ref{assrie}, let $1 \le i,j \le d$, and define
\[
\bigl\langle S^i, S^j\bigr\rangle_t :=
S^i_t S^j_t -
S^i_0 S^j_0 - \int
_0^t S^i_r \,\dd
S^j_r - \int_0^t
S^j_r \,\dd S^i_r.
\]
Then $\langle S^i,S^j\rangle$ is a continuous function and
%
\begin{equation}
\label{eq:quadratic variation as limit} \bigl\langle S^i, S^j\bigr
\rangle_t = \lim_{n \too\infty} \bigl\langle
S^i, S^j \bigr\rangle^n_t = \lim
_{n \too\infty} \sum_{k=0}^{N_n-1}
\bigl(S^i_{t^n_{k+1}
\wedge t} - S^i_{t^n_k \wedge t}\bigr)
\bigl(S^j_{t^n_{k+1} \wedge t} - S^j_{t^n_k
\wedge t}\bigr).
\end{equation}
The sequence $(\langle S^i, S^j \rangle^n)_n$ is of uniformly bounded
total variation, and in particular $\langle S^i, S^j\rangle$ is of
bounded variation. We write $\langle S\rangle= \langle S,S\rangle=
(\langle S^i, S^j\rangle)_{1 \le i,j \le d}$, and call $\langle S\rangle
$ the \emph{quadratic variation} of $S$.
\end{lem}

\begin{pf}
The function $\langle S^i, S^j\rangle$ is continuous by definition.
The specific form~\eqref{eq:quadratic variation as limit} of $\langle
S^i, S^j\rangle$ follows from two simple observations:
\[
S^i_t S^j_t -
S^i_0 S^j_0 = \sum
_{k=0}^{N_n-1} \bigl(S^i_{t^n_{k+1}
\wedge t}S^j_{t^n_{k+1} \wedge t}
- S^i_{t^n_k \wedge t}S^j_{t^n_k
\wedge t} \bigr)
\]
for every $n \in\N$, and
\[
S^i_{t^n_{k+1} \wedge t}S^j_{t^n_{k+1} \wedge t} -
S^i_{t^n_k \wedge
t}S^j_{t^n_k \wedge t} =
S^i_{t^n_{k} \wedge t} S^j_{t^n_{k} \wedge t,
t^n_{k+1} \wedge t} +
S^j_{t^n_{k} \wedge t} S^i_{t^n_{k} \wedge t,
t^n_{k+1} \wedge t} +
S^i_{t^n_{k} \wedge t,t^n_{k+1} \wedge t} S^j_{t^n_{k} \wedge t, t^n_{k+1} \wedge t},
\]
so that the convergence in~\eqref{eq:quadratic variation as limit} is a
consequence of the convergence of $(S^n \cdot S)$ to $\int S \,\dd S$.

To see that $\langle S^i, S^j\rangle$ is of bounded variation, note that
\[
S^i_{t^n_{k} \wedge t,t^n_{k+1} \wedge t} S^j_{t^n_{k} \wedge t,
t^n_{k+1}\wedge t} =
\tfrac{1}{4} \bigl( \bigl(\bigl(S^i + S^j
\bigr)_{t^n_{k}
\wedge t, t^n_{k+1}\wedge t} \bigr)^2 - \bigl(\bigl(S^i -
S^j\bigr)_{t^n_{k}
\wedge t, t^n_{k+1}\wedge t} \bigr)^2 \bigr)
\]
(read $\langle S^i,S^j \rangle= 1/4(\langle S^i + S^j\rangle- \langle
S^i - S^j\rangle)$). In other words, the $n$th approximation of
$\langle S^i,S^j\rangle$ is the difference of two increasing functions,
and its total variation is bounded from above by
\[
\sum_{k=0}^{N_n-1} \bigl( \bigl(
\bigl(S^i + S^j\bigr)_{t^n_{k}, t^n_{k+1}}
\bigr)^2 + \bigl(\bigl(S^i - S^j
\bigr)_{t^n_{k}, t^n_{k+1}} \bigr)^2 \bigr) \lesssim \sup
_m \sum_{k=0}^{N_m-1}
\bigl(\bigl(S^i_{t^m_{k}, t^m_{k+1}}\bigr)^2 +
\bigl(S^j_{t^m_{k}, t^m_{k+1}}\bigr)^2 \bigr).
\]
Since the right-hand side is finite, also the limit $\langle
S^i,S^j\rangle$ is of bounded variation.
\end{pf}

Given the quadratic variation, the existence of the Stratonovich
integral is straightforward:

\begin{lem}\label{lem:stratonovich}
Under Assumption~\ref{assrie}, define $\tilde
{S}^n|_{[t^n_k,t^n_{k+1}]}$ as the linear interpolation of $S_{t^n_k}$
and $S_{t^n_{k+1}}$ for $k=0,\ldots, N_n-1$. Then $(\int\tilde{S}^n \,\dd
\tilde{S}^n)$ converges uniformly to
%
\begin{equation}
\label{eq:stratonovich vs ito} \int_s^t S_r
\circ\,\dd S_r := \int_s^t
S_r \,\dd S_r + \frac{1}{2} \langle S
\rangle_{s,t}.
\end{equation}
Moreover, setting $\tilde{A}^n(s,t) := \int_s^t \tilde{S}^n_{s,r} \,\dd
\tilde{S}^n_r$ for $(s,t) \in\Delta_T$, we have $\sup_n \| \tilde
{A}^n\|_{p/2\mbox{-}\var} < \infty$.
\end{lem}

\begin{pf}
Let $n \in\N$ and $k \in\{0, \ldots, N_n-1\}$. Then for $t \in
[t^n_k,t^n_{k+1}]$ we have
\[
\tilde{S}^n_t = S_{t^n_k} + \frac{t - t^n_k}{t^n_{k+1} - t^n_k}
S_{t^n_k, t^n_{k+1}},
\]
so that
%
\begin{equation}
\label{eq:stratonovich pr1} \int_{t^n_k}^{t^n_{k+1}}
\tilde{S}^n_r \,\dd\tilde{S}^n_r =
S_{t^n_k} S_{t^n_k, t^n_{k+1}} + \frac{1}{2} S_{t^n_k, t^n_{k+1}}
S_{t^n_k, t^n_{k+1}},
\end{equation}
from where the uniform convergence and the representation~\eqref
{eq:stratonovich vs ito} follow by Lemma~\ref{lem:quadratic variation}.

To prove that $\tilde{A}^n$ has uniformly bounded $\frac
{p}{2}$-variation, consider $(s,t) \in\Delta_T$. If there exists $k$
such that $t^n_k \le s < t \le t^n_{k+1}$, then we estimate
%
\begin{eqnarray}
\label{eq:stratonovich pr2}
\bigl|\tilde{A}^n(s,t)\bigr|^{p/2} &=& \biggl| \int
_s^t \tilde{S}^n_{s,r}
\,\dd\tilde {S}^n_r \biggr|^{p/2} \le\biggl | \int
_s^t (r-s) \frac
{|S_{t^n_k,t^n_{k+1}}|^2}{|t^n_{k+1} - t^n_k|^2} \,\dd r
\biggr|^{p/2}
\nonumber
\\[-8pt]
\\[-8pt]
\nonumber
& = &\frac{1}{2^{p/2}} |t-s|^p \frac
{|S_{t^n_k,t^n_{k+1}}|^p}{|t^n_{k+1} - t^n_k|^p} \le
\frac
{|t-s|}{|t^n_{k+1} - t^n_k|} \|S\|_{p\mbox{-}\var,[t^n_k,t^n_{k+1}]}^p.
\end{eqnarray}
Otherwise, let $k_0$ be the smallest $k$ such that $t^n_k \in(s,t)$,
and let $k_1$ be the largest such $k$. We decompose
\[
\tilde{A}^n(s,t) = \tilde{A}^n\bigl(s,t^n_{k_0}
\bigr) + \tilde {A}^n\bigl(t^n_{k_0},t^n_{k_1}
\bigr) + \tilde{A}^n\bigl(t^n_{k_1},t\bigr) +
\tilde {S}^n_{s,t^n_{k_0}} \tilde{S}^n_{t^n_{k_0}, t^n_{k_1}} +
\tilde {S}^n_{s,t^n_{k_1}} \tilde{S}^n_{t^n_{k_1}, t}.
\]
We get from~\eqref{eq:stratonovich pr1} that
\[
\bigl|\tilde{A}^n\bigl(t^n_{k_0},t^n_{k_1}
\bigr)\bigr|^{p/2} \lesssim\bigl| \bigl(S^n \cdot S
\bigr)_{t^n_{k_0},t^n_{k_1}} - S_{t^n_{k_0}} S_{t^n_{k_0}, t^n_{k_1}}\bigr |^{p/2} +
\bigl(\langle S \rangle^n_{t^n_{k_0},t^n_{k_1}}\bigr)^{p/2},
\]
where $\langle S \rangle^n$ denotes the $n$th approximation of the
quadratic variation. By the assumption \ref{assrie} and Lemma~\ref
{lem:quadratic variation}, there exists a control function $\tilde{c}$
so that the right-hand side is bounded from above by $\tilde
{c}(t^n_{k_0},t^n_{k_1})$. Combining this with~\eqref{eq:stratonovich
pr2} and a simple estimate for the terms $\tilde{S}^n_{s,t^n_{k_0}}
\tilde{S}^n_{t^n_{k_0}, t^n_{k_1}}$ and $\tilde{S}^n_{s,t^n_{k_1}}
\tilde{S}^n_{t^n_{k_1}, t}$, we deduce that $\| \tilde{A}^n \|_{p/2\mbox
{-}\var} \lesssim\tilde{c}(0,T) + \|S\|_{p\mbox{-}\var}^2$, and the
proof is complete.
\end{pf}

We are now ready to prove the main result of this section.

\begin{thmm}\label{thmm:riemann sum}
Under Assumption~\ref{assrie}, let $q > 0$ be such that $2/p + 1/q >
1$. Let $(F,F') \in\CC_\BS^{q}$ be a controlled path such that $F$ is
continuous. Then the rough path integral $\int F \,\dd S$ which was
defined in Theorem~\ref{thmm:rough path integral} is given by
\[
\int_0^t F_s \,\dd
S_s = \lim_{n\too\infty} \sum
_{k=0}^{N_n-1} F_{t^n_k} S_{t^n_k \wedge t, t^n_{k+1} \wedge t},
\]
where the convergence is uniform in $t$.
\end{thmm}

\begin{pf}
For $n \in\N$ define $\tilde{F}^n$ as the linear interpolation of $F$
between the points in $\pi^n$. Then $(\tilde{F}^n, F')$ is controlled
by $\tilde{S}^n$: Clearly, $\|\tilde{F}^n\|_{q\mbox{-}\var} \le\| F\|
_{q\mbox{-}\var}$. The remainder $\tilde{R}^n_{\tilde{F}^n}$ of $\tilde
{F}^n$ with respect to $\tilde{S}^n$ is given by $\tilde{R}^n_{\tilde
{F}^n}(s,t) = \tilde{F}^n_{s,t} - F'_s \tilde{S}^n_{s,t}$ for $(s,t)
\in\Delta_T$. We need to show that $\tilde{R}^n_{\tilde{F}^n}$ has
finite $r$-variation for $1/r = 1/p + 1/q$.

If $t^n_k \le s \le t \le t^n_{k+1}$, we have
%
\begin{eqnarray}
\label{eq:riemann sum pr1}
\nonumber
\bigl|\tilde{R}^n_{\tilde{F}^n}(s,t)\bigr|^r
& =& \biggl| \frac{t-s}{t^n_{k+1} -
t^n_k} F_{t^n_k,t^n_{k+1}} - F'_s
\frac{t-s}{t^n_{k+1} - t^n_k} S_{t^n_k,t^n_{k+1}} \biggr|^r
\\
& \le& \biggl|\frac{t-s}{t^n_{k+1} - t^n_k} \biggr|^r \bigl( \| R_F
\|_{r\mbox
{-}\var,[t^n_k, t^n_{k+1}]} + \bigl\|F'\bigr\|_{q\mbox{-}\var,[t^n_k, s]}^{r/q} \| S
\|_{p\mbox{-}\var,[t^n_k,t^n_{k+1}]}^{r/p} \bigr)
\\
& \le&\frac{|t-s|}{|t^n_{k+1} - t^n_k|} \bigl( \| R_F\|_{r\mbox{-}\var
,[t^n_k, t^n_{k+1}]} +
\bigl\|F'\bigr\|_{q\mbox{-}\var,[t^n_k, t^n_{k+1}]} + \| S\|_{p\mbox{-}\var,[t^n_k,t^n_{k+1}]} \bigr),\nonumber
\end{eqnarray}
where in the last step we used that $1/r = 1/p + 1/q$, and thus $r/q +
r/p =1$.

Otherwise, there exists $k \in\{1, \ldots, N_n-1\}$ with $t^n_k \in
(s,t)$. Let $k_0$ and $k_1$ the smallest and largest such $k$,
respectively. Then
\begin{eqnarray*}
\bigl|\tilde{R}^n_{\tilde{F}^n}(s,t)\bigr|^r
&\lesssim_r &\bigl|\tilde{R}^n_{\tilde
{F}^n}
\bigl(s,t^n_{k_0}\bigr)\bigr|^r + \bigl|
\tilde{R}^n_{\tilde{F}^n}\bigl(t^n_{k_0},
t^n_{k_1}\bigr)\bigr|^r + \bigl|\tilde{R}^n_{\tilde{F}^n}
\bigl(t^n_{k_1}, t\bigr)\bigr|^r \\
&&{}+
\bigl|F'_{s,t^n_{k_0}} S_{t^n_{k_0}, t^n_{k_1}}\bigr|^r +
\bigl|F'_{s,t^n_{k_1}} S_{t^n_{k_1}, t}\bigr|^r.
\end{eqnarray*}
Now $\tilde{R}^n_{\tilde{F}^n}(t^n_{k_0}, t^n_{k_1}) = R_F(t^n_{k_0},
t^n_{k_1})$, and therefore we can use~\eqref{eq:riemann sum pr1}, the
assumption on $R_F$, and the fact that $1/r = 1/p + 1/q$ (which is
needed to treat the last two terms on the right-hand side), to obtain
\[
\bigl\|\tilde{R}^n_{\tilde{F}^n}\bigr\|_{r\mbox{-}\var}
\lesssim_r \| R_F\| _{r\mbox{-}\var} +
\bigl\|F'\bigr\|_{q\mbox{-}\var} + \| S\|_{p\mbox{-}\var}.
\]
On the other side, since $F$ and $R_F$ are continuous, $(\tilde{F}^n,
\tilde{R}^n_{\tilde{F}^n})$ converges uniformly to $(F,R_F)$. Now for
continuous functions, uniform convergence with uniformly bounded
$p$-variation implies convergence in $p'$-variation for every $p' > p$.
See Exercise~2.8 in~\cite{Friz2014} for the case of H\"older continuous
functions.

Thus, using Lemma~\ref{lem:stratonovich}, we see that if $p' > p$ and
$q' > q$ are such that $2/p' + 1/q' > 0$, then $((\tilde{S}^n, \tilde
{A}^n)_n)$ converges in $(p', p'/2)$-variation to $(S, A^\circ)$, where
$A^\circ(s,t) = A(s,t) + 1/2 \langle S \rangle_{s,t}$. Similarly,
$((\tilde{F}^n, F', \tilde{R}^n_{\tilde{F}^n}))$ converges in $(q', p',
r')$-variation to $(F, F', R_F)$, where $1/r' = 1/p' + 1/q'$.

Proposition~\ref{prop:continuous rough path integral} now yields the
uniform convergence of $\int\tilde{F}^n \,\dd\tilde{S}^n$ to $\int F
\circ\,\dd S$, by which we denote the rough path integral of the
controlled path $(F, F')$ against the rough path $(S, A^\circ)$. But
for every $t \in[0,T]$, we have
\begin{eqnarray*}
\lim_{n\too\infty} \int_0^t
\tilde{F}^n_s \,\dd\tilde{S}^n_s &
=& \lim_{n\too\infty} \sum_{k: t^n_{k+1} \le t}
\frac{1}{2} (F_{t^n_k} + F_{t^n_{k+1}}) S_{t^n_k, t^n_{k+1}}
\\
& =& \lim_{n\too\infty} \biggl( \sum_{k: t^n_{k+1} \le t}
F_{t^n_k} S_{t^n_k, t^n_{k+1}} + \frac{1}{2} \sum
_{k: t^n_{k+1} \le t} F_{t^n_k,
t^n_{k+1}} S_{t^n_k, t^n_{k+1}} \biggr).
\end{eqnarray*}
Using that $F$ is controlled by $S$, it is easy to see that the second
term on the right-hand side converges uniformly to $1/2 \int_0^t F'_s
\,\dd\langle S \rangle_s$, $t \in[0,T]$. Thus, the Riemann sums $\sum_k
F_{t^n_k} S_{t^n_k \wedge\cdot, t^n_{k+1} \wedge\cdot}$ converge
uniformly to $\int F \circ\,\dd S - 1/2 \int F' \,\dd\langle S\rangle$,
and from the representation of the rough path integral as limit of
compensated Riemann sums~\eqref{eq:compensated riemann sums}, it is
easy to see that $\int F\circ\,\dd S = \int F \,\dd S + 1/2 \int F' \,\dd
\langle S \rangle$, which completes the proof.
\end{pf}

\begin{rmk}
Given Theorem~\ref{thmm:riemann sum} it is natural to conjecture that
if $(S, A)$ is the rough path which we constructed in Theorem~\ref
{thmm:area variation} and Lemma~\ref{lem:uniformly bounded variation},
then for typical price paths and for adapted, controlled, and
continuous integrands $F$ the rough path integral agrees with the model
free integral of Section~\ref{sec:cpi}. This seems not very easy to
show, but what can be verified is that if $F \in C^{1+\varepsilon}$,
then for the integrand $F(S)$ both integrals coincide -- simply take
Riemann sums along the dyadic stopping times defined in~\eqref
{eq:dyadic stopping times}.
\end{rmk}

Theorem~\ref{thmm:riemann sum} is reminiscent of F\"ollmer's pathwise
It\^o integral~\cite{Follmer1981}. F\"ollmer assumes that the quadratic
variation $\langle S \rangle$ of $S$ exists along a given sequence of
partitions and is continuous, and uses this to prove an It\^o formula
for $S$: if $F \in C^2$, then
%
\begin{equation}
\label{eq:foellmer ito} F(S_t) = F(S_0) + \int
_0^t \nabla F(S_s) \,\dd
S_s + \frac{1}{2} \int_0^t
\mathrm{D}^2 F(S_s) \,\dd\langle S \rangle_s,
\end{equation}
where the integral $\int_0^\cdot\nabla F(S_s) \,\dd S_s$ is given as
limit of Riemann sums along that same sequence of partitions. Friz and
Hairer~\cite{Friz2014} observe that if for $p\in(2,3)$ the function
$S$ is of finite $p$-variation and $\langle S \rangle$ is an arbitrary
continuous function of finite $p/2$-variation, then setting
\[
\mathrm{Sym}(A) (s,t) := \tfrac{1}{2} \bigl( S^i_{s,t}
S^j_{s,t} + \langle S \rangle_{s,t}\bigr)
\]
one obtains a ``reduced rough path'' $(S,\mathrm{Sym}(A))$. They
continue to show that if $F$ is controlled by $S$ with \emph{symmetric}
derivative $F'$, then it is possible to define the rough path integral
$\int F \,\dd S$. This is not surprising since then we have $F'_s A_{s,t}
= F'_s \mathrm{Sym}(A)_{s,t}$ for the compensator term in the
definition of the rough path integral. They also derive an It\^o
formula for reduced rough paths, which takes the same form as~\eqref
{eq:foellmer ito}, except that now $\int\nabla F(S) \,\dd S$ is a rough
path integral (and therefore defined as limit of compensated Riemann sums).

So both the assumption and the result of~\cite{Friz2014} are slightly
different from the ones in~\cite{Follmer1981}, and while it seems
intuitively clear, it is still not shown rigorously that F\"ollmer's
pathwise It\^o integral is a special case of the rough path integral.
We will now show that F\"ollmer's result is a special case of
Theorem~\ref{thmm:riemann sum}. For that purpose, we only need to prove
that F\"ollmer's condition on the convergence of the quadratic
variation is a special case of the assumption in Theorem~\ref
{thmm:riemann sum}, at least as long as we only need the symmetric part
of the area.

\begin{definition}
Let $f \in C([0,T],\mathbb{R})$ and let $\pi^n = \{0 = t^n_0 < t^n_1 <
\cdots< t^n_{N_n} = T\}$, $n \in\N$ be such that $\sup\{ |
f_{t^n_k,t^n_{k+1}}|: k=0, \ldots, N_n-1\}$ converges to 0. We say that
$f$ has \emph{quadratic variation along $(\pi^n)$ in the sense of F\"ollmer}
if the sequence of discrete measures $(\mu^n)$ on $([0,T], \mathcal
{B}[0,T])$, defined by
%
\begin{equation}
\label{eq:follmer discrete measure} \mu_n := \sum_{k=0}^{N_n-1}
\vert f_{t^n_k,t^n_{k+1}}\vert^2 \delta_{t^n_k},
\end{equation}
converges weakly to a non-atomic measure $\mu$. We write $[ f ]_t$ for
the ``distribution function'' of $\mu$ (in general $\mu$ will not be a
probability measure). The function $f=(f^1, \ldots,f^d) \in
C([0,T],\mathbb{R}^d)$ has \emph{quadratic variation along $(\pi^n)$ in
the sense of F\"ollmer} if this holds for all $f^i$ and $f^i+f^j$,
$1\leq i,j \leq d$. In this case, we set
\[
\bigl[ f^i, f^j \bigr]_t :=
\tfrac{1}{2} \bigl(\bigl[ f^i + f^j
\bigr]_t - \bigl[ f^i \bigr]_t - \bigl[
f^j \bigr]_t\bigr),\qquad t \in[0,T].
\]
\end{definition}

\begin{lem}[(see also~\cite{Vovk2011}, Proposition~6.1)]\label{lem:var}
Let $p \in(2,3)$, and let $S=(S^1, \ldots,S^d) \in C([0,T], \mathbb
{R}^d)$ have finite $p$-variation. Let $\pi^n = \{0 = t^n_0 < t^n_1 <
\cdots< t^n_{N_n} = T\}$, $n \in\N$, be a sequence of partitions such
that $\sup\{ | S_{t^n_k,t^n_{k+1}}|: k=0, \ldots, N_n-1\}$ converges to
0. Then the following conditions are equivalent:
\begin{longlist}[(1)]
\item[(1)] The function $S$ has quadratic variation along $(\pi^n)$ in the
sense of F\"ollmer.

\item[(2)] For all $1\leq i,j \leq d$, the discrete quadratic variation
\[
\bigl\langle S^i, S^j \bigr\rangle^n_t
:= \sum_{k=0}^{N_n-1} S^i_{t^n_k\wedge t,
t^n_{k+1}\wedge t}
S^j_{t^n_k\wedge t, t^n_{k+1}\wedge t}
\]
converges uniformly in $C([0,T],\mathbb{R})$ to a limit $\langle
S^i,S^j\rangle$.

\item[(3)] For $S^{n,i} := \sum_{k=0}^{N_n - 1} S^i_{t^n_k} \1_{[t^n_k,
t^n_{k+1})}$, $i \in\{1, \ldots, d\}$, $n \in\N$, the Riemann sums
$(S^{n,i} \cdot S^{j})+(S^{n,j} \cdot S^{i})$ converge uniformly to a
limit $\int S^i \,\dd S^j+ \int S^j \,\dd S^i$. Moreover, the symmetric
part of the approximate area,
\[
\mathrm{Sym}\bigl(A^n\bigr)^{i,j}(s,t)= \tfrac{1}{2}
\bigl( \bigl(S^{n,i} \cdot S^{j}\bigr)_{s,t} +
\bigl(S^{n,j} \cdot S^{i}\bigr)_{s,t} -
S^{i}_s S^{j}_{s,t} -
S^{j}_s S^{i}_{s,t} \bigr),\quad 1\le i,j
\le d, (s,t) \in\Delta_T,
\]
has uniformly bounded $p/2$-variation along $(\pi^n)$, in the sense
of~\eqref{eq:bounded p/2 variation assumption}.
\end{longlist}
If these conditions hold, then $[S^i,S^j] =\langle S^i,S^j \rangle$ for
all $1\leq i,j\leq d$.
\end{lem}

\begin{pf}
Assume (1) and note that
\[
S^i_{t^n_k\wedge t, t^n_{k+1}\wedge t} S^j_{t^n_k\wedge t,
t^n_{k+1}\wedge t} =
\tfrac{1}{2} \bigl( \bigl(\bigl(S^i + S^j
\bigr)_{t^n_k\wedge t,
t^n_{k+1}\wedge t}\bigr)^2 - \bigl(S^i_{t^n_k\wedge t, t^n_{k+1}\wedge t}
\bigr)^2 - \bigl(S^j_{t^n_k\wedge t, t^n_{k+1}\wedge t}\bigr)^2
\bigr).
\]
Thus, the uniform convergence of $\langle S^i, S^j \rangle^n$ and the
fact that $\langle S^i, S^j\rangle= [S^i, S^j]$ follow once we show
that F\"ollmer's weak convergence of the measures~\eqref{eq:follmer
discrete measure} implies the uniform convergence of their distribution
functions. But since the limiting distribution is continuous by
assumption, this is a standard result.

Next, assume (2). The uniform convergence of the Riemann sums $(S^{n,i}
\cdot S^{j})+(S^{n,j} \cdot S^{i})$ is shown as in Lemma~\ref
{lem:quadratic variation}. To see that $\mathrm{Sym}(A^n)$ has
uniformly bounded $p/2$-variation along $(\pi^n)$, note that for $0 \le
k \le\ell\le N_n$ and $1 \le i,j \le d$ we have
\begin{eqnarray*}
&&\bigl|\bigl(S^{n,i} \cdot S^j\bigr)_{t^n_k, t^n_\ell} +
\bigl(S^{n,j} \cdot S^i\bigr)_{t^n_k,
t^n_\ell} -
S^{i}_s S^{j}_{t^n_k, t^n_\ell} -
S^{j}_s S^{i}_{t^n_k,
t^n_\ell}\bigr|^{p/2}
\\
&&\quad= \bigl|S^i_{t^n_k, t^n_\ell} S^j_{t^n_k, t^n_\ell} - \bigl\langle
S^i, S^j \bigr\rangle^n_{t^n_k, t^n_\ell}\bigr|^{p/2}
\\
&&\quad\le\| S\|_{p\mbox{-}\var, [t^n_k, t^n_\ell]} + \bigl\| \bigl\langle S^i, S^j
\bigr\rangle^n \bigr\|_{1\mbox{-}\var,[t^n_k, t^n_\ell]}.
\end{eqnarray*}
That $ \| \langle S^i, S^j \rangle^n \|_{1\mbox{-}\var}$ is uniformly
bounded in $n$ is shown in Lemma~\ref{lem:quadratic variation}.

That (3) implies (1) is also shown in Lemma~\ref{lem:quadratic variation}.
\end{pf}

\begin{rmk}
With Theorem~\ref{thmm:riemann sum} we can only derive an It\^o
formula for $F \in C^{2+\varepsilon}$, since we are only able to
integrate $\nabla F(S)$ if $\nabla F \in C^{1+\varepsilon}$. But this
only seems to be due to the fact that our analysis is not sharp. We
expect that typical price paths have an associated rough path of finite
$2$-variation, up to logarithmic corrections. For such rough paths, the
integral extends to integrands $F \in C^1$, see Chapter~10.5 of~\cite
{Friz2010}. For typical price paths (but not for the area), it is shown
in~\cite{Vovk2012}, Section~4.3, that they are of finite $2$-variation
up to logarithmic corrections.
\end{rmk}

\begin{appendix}\label{app}

\section{Pathwise Hoeffding inequality}\label{a:hoeffding}

In the construction of the pathwise It\^{o} integral for typical price
processes, we needed the following result, a pathwise formulation of
the Hoeffding inequality which is due to Vovk. Here we present a
slightly adapted version.

\begin{lem}[(\cite{Vovk2012}, Theorem A.1)]\label{l:hoeffding}
Let $(\tau_n)_{n\in\N}$ be a strictly increasing sequence of stopping
times with $\tau_0 = 0$, such that for every $\omega\in\Omega$ we
have $\tau_n(\omega) = \infty$
for all but finitely many $n \in\N$. Let for $n \in\N$ the function
$h_n\colon\Omega\rightarrow\R^d$ be $\mathcal{F}_{\tau
_n}$-measurable, and suppose that there exists
a $\mathcal{F}_{\tau_n}$-measurable bounded function $b_n\colon\Omega
\rightarrow\R$, such that
%
\renewcommand{\theequation}{\arabic{equation}}
\setcounter{equation}{19}
\begin{equation}
\label{e:hoeffding condition} \sup_{t \in[0,T]} \bigl\vert h_n(\omega)
S_{\tau_{n}\wedge t, \tau
_{n+1}\wedge t}(\omega)\bigr \vert\le b_n(\omega)
\end{equation}
for all $\omega\in\Omega$. Then for every $\lambda\in\R$ there
exists a simple strategy $H^\lambda\in\mathcal{H}_1$ such that
\[
1 + \bigl(H^\lambda\cdot S\bigr)_t \ge\exp \Biggl( \lambda
\sum_{n=0}^{\infty} h_n
S_{\tau_{n}\wedge t, \tau_{n+1}\wedge t} - \frac{\lambda^2}{2}\sum_{n=0}^{N_t}b_n^2
\Biggr)
\]
for all $t \in[0,T]$, where $N_t:= \max\{n \in\N: \tau_n \le t\}$.
\end{lem}

\begin{pf}
Let $\lambda\in\R$. The proof is based on the following
deterministic inequality: if \eqref{e:hoeffding condition} is
satisfied, then for all $\omega\in\Omega$ and all $t \in[0,T]$ we
have that
%
\renewcommand{\theequation}{\arabic{equation}}
\setcounter{equation}{20}
\begin{eqnarray}
\label{e:hoeffing proof}
\nonumber
&&\exp \biggl(\lambda h_n(\omega)
S_{\tau_{n}\wedge t, \tau_{n+1}\wedge
t}(\omega) - \frac{\lambda^2}{2} b^2_n(
\omega) \biggr) - 1
\\
&&\quad \le\exp \biggl(- \frac{\lambda^2}{2} b^2_n(
\omega) \biggr) \frac{e^{\lambda b_n(\omega)}-e^{-\lambda b_n(\omega)}}{2b_n(\omega
)}h_n(\omega) S_{\tau_{n}\wedge t, \tau_{n+1}\wedge t}(
\omega)
\\
&&\quad =: f_n(\omega) S_{\tau_{n}\wedge t, \tau_{n+1}\wedge
t}(\omega).\nonumber
\end{eqnarray}
This inequality is shown in~(A.1) of~\cite{Vovk2012}. We define
$H^\lambda_t := \sum_{n} F_n 1_{(\tau_n, \tau_{n+1}]}(t)$, with $F_n$
that have to be specified. We choose $F_0 := f_0$, which is bounded and
$\F_{\tau_0}$-measurable, and on $[0,\tau_1]$ we obtain
\[
1 + \bigl(H^\lambda\cdot S\bigr)_t \ge\exp \biggl(\lambda
h_0 S_{\tau_{n}\wedge t,
\tau_{n+1}\wedge t} - \frac{\lambda^2}{2} b^2_0
\biggr).
\]
Observe also that $1 + (H^\lambda\cdot S)_{\tau_1} = 1 + f_0S_{\tau
_0,\tau_1}$ is bounded, because by assumption $h_0 S_{\tau_0,\tau_1}$
is bounded by the bounded random variable $b_0$.

Assume now that $F_k$ has been defined for $k = 0, \ldots, m-1$, that
\[
1 + \bigl(H^\lambda\cdot S\bigr)_t \ge\exp \Biggl( \lambda
\sum_{n=0}^{\infty} h_n
S_{\tau_{n}\wedge t, \tau_{n+1}\wedge t} - \frac{\lambda^2}{2}\sum_{n=0}^{N_t}b_n^2
\Biggr)
\]
for all $t \in[0, \tau_m]$, and that $1 + (H^\lambda\cdot S)_{\tau
_m}$ is bounded. We define $F_m := (1 + (H^\lambda\cdot S)_{\tau
_m})f_m$, which is $\F_{\tau_m}$-measurable and bounded. From \eqref
{e:hoeffing proof}, we obtain for $t \in[\tau_m, \tau_{m+1}]$
\begin{eqnarray*}
&&1 + \bigl(H^\lambda\cdot S\bigr)_t\\
&&\quad = 1 +
\bigl(H^\lambda\cdot S\bigr)_{\tau_m} + \bigl(1 +
\bigl(H^\lambda\cdot S\bigr)_{\tau_m}\bigr) f_m
S_{\tau_{m}\wedge t, \tau_{m+1}\wedge
t}
\\
&&\quad\ge\bigl(1 + \bigl(H^\lambda\cdot S\bigr)_{\tau_m}\bigr) \exp
\biggl(\lambda h_m S_{\tau
_{m}\wedge t, \tau_{m+1}\wedge t} - \frac{\lambda^2}{2}
b^2_m \biggr)
\\
&&\quad\ge\exp \Biggl( \lambda\sum_{n=0}^{\infty}
h_n S_{\tau_{n}\wedge t,
\tau_{n+1}\wedge t} - \frac{\lambda^2}{2}\sum
_{n=0}^{N_t}b_n^2 \Biggr),
\end{eqnarray*}
where in the last step we used the induction hypothesis. From the first
line of the previous equation, we also obtain that $1 + (H^\lambda
\cdot S)_{\tau_{m+1}}$ is bounded
because $f_m S_{\tau_{m}, \tau_{m+1}}$ is bounded for the same reason
that $f_0 S_{\tau_{0}, \tau_{1}}$ is bounded.
\end{pf}

\section{Davie's criterion}\label{a:davie}

It was already observed by Davie~\cite{Davie2007} that in certain
situations the rough path integral can be constructed as limit of
Riemann sums and not just compensated Riemann sums. Davie shows that
under suitable conditions, the usual Euler scheme (without ``area
compensation'') converges to the solution of a given rough differential
equation. But from there it is easily deduced that then also the rough
path integral is given as limit of Riemann sums. Here we show that
Davie's criterion implies our assumption \ref{assrie}.

Let $p \in(2,3)$ and let $\mathbb{S}=(S,A)$ be a $1/p$-H\"older
continuous rough path, that is $|S_{s,t}| \lesssim|t-s|^{1/p}$ and
$|A(s,t)| \lesssim|t-s|^{2/p}$. Write $\alpha:= 1/p$ and let $\beta
\in(1-\alpha, 2 \alpha)$. Davie assumes that there exists $C > 0$ such
that the area process $A$ satisfies
%
\begin{equation}
\label{eq:davie} \Biggl\vert\sum_{j=k}^{\ell-1} A
\bigl(jh,(j+1)h\bigr) \Biggr\vert\le C (\ell -k)^\beta h^{2 \alpha},
\end{equation}
whenever $0 < k < \ell$ are integers and $h>0$ such that $\ell h \leq
T$. Under these conditions, Theorem~7.1 of~\cite{Davie2007} implies
that for $F \in C^\gamma$ with $\gamma> p$ and for $t^n_k = kT/n$,
$n,k \in\N$, the Riemann sums
\[
\sum_{k=0}^{n-1} F(S_{t^n_k})
S_{t^n_k \wedge t, t^n_{k+1} \wedge t}, \quad t \in[0,T],
\]
converge uniformly to the rough path integral. But it can be easily
deduced from~\eqref{eq:davie} that the area process $A$ is given as
limit of non-anticipating Riemann sums along $(t^n)_n$. Indeed, letting
$h = T/n$,
\begin{eqnarray*}
&&\Biggl \vert\int_0^t S_s \,\dd
S_s- \sum_{k=0}^{n-1}
S_{t^n_k}S_{t^n_k
\wedge t, t^n_{k+1}\wedge t} \Biggr\vert\\[-2pt]
&&\quad= \Biggl\vert\sum_{k=0}^{n-1}
\biggl( \int_{t^n_k \wedge t}^{t^n_{k+1} \wedge t} S_s \,\dd
S_s - S_{t^n_k \wedge t}S_{t^n_k \wedge t, t^n_{k+1}\wedge t} \biggr)\Biggr \vert
\\[-2pt]
&&\quad = \Biggl| \sum_{k=0}^{n-1} A
\bigl(t^n_k \wedge t, t^n_{k+1}
\wedge t\bigr) \Biggr| \le \Biggl| \sum_{k=0}^{\lfloor t/h\rfloor- 1}
A_{kh, (k+1)h}\Biggr | + \bigl|A\bigl(\lfloor t/h\rfloor, t\bigr)\bigr|
\\[-2pt]
& &\quad\lesssim C \lfloor t/h \rfloor^\beta h^{2\alpha} +
h^{2\alpha}\|A\|_{2\alpha} \lesssim C t h^{2\alpha- \beta} +
h^{2\alpha
} \|A\|_{2\alpha}.
\end{eqnarray*}
Since $\beta< 2\alpha$, the right-hand side converges to 0 as $n$ goes
to $\infty$ (and thus $h$ goes to 0). Furthermore,~\eqref{eq:davie}
implies the ``uniformly bounded $p/2$-variation'' condition~\eqref
{eq:bounded p/2 variation assumption}:
\begin{eqnarray*}
\bigl\vert\bigl(S^n \cdot S\bigr)_{t^n_k,t^n_\ell} -
S_{t^n_k}S_{t^n_k,
t^n_\ell} \bigr\vert &\leq& \biggl\vert\int_{t^n_k}^{t^n_\ell}
S_s \,\dd S_s - S_{t^n_k}S_{t^n_k, t^n_\ell} \biggr\vert+
\Biggl\vert \sum_{j=k}^{\ell
-1} \biggl( \int
_{t^n_j}^{t^n_{j+1}} S_s \,\dd S_s -
S_{t^n_j} S_{t^n_j,t^n_{j+1}} \biggr) \Biggr\vert
\\[-2pt]
& \le&\| A \|_{2 \alpha} \bigl\vert t^n_\ell-
t^n_k \bigr\vert ^{2\alpha} + \Biggl\vert\sum
_{j=k}^{\ell-1} A_{t^n_{k},t^n_{k+1}} \Biggr\vert\\[-2pt]
&\le&\| A
\|_{2 \alpha}\bigl \vert t^n_\ell- t^n_k
\bigr\vert^{2\alpha} + C (\ell-k)^\beta h^{2\alpha}
\\[-2pt]
& \le& \| A \|_{2 \alpha} \bigl\vert t^n_\ell-
t^n_k\bigr \vert ^{2\alpha} + C \bigl|t^n_\ell-
t^n_k\bigr|^{2\alpha}.
\end{eqnarray*}
\end{appendix}\vspace*{-5pt}

\section*{Acknowledgements}
N. Perkowski was supported by the Fondation Sciences Math\'ematiques de Paris
(FSMP) and by a public grant overseen by the French National Research
Agency (ANR) as part of the ``Investissements d'Avenir'' program
(reference: ANR-10-LABX-0098). The main part of the research was
carried out while N. Perkowski was employed by Humboldt-Universit\"at zu Berlin.
D.J. Pr\"omel is supported by a Ph.D. scholarship of the DFG Research Training
Group 1845 ``Stochastic Analysis with Applications in Biology, Finance
and Physics''.

%

\printhistory
\end{document}